\documentclass[11pt,reqno]{amsart}
\usepackage{amsmath,amsfonts,amssymb,amscd,amsthm,amsbsy,epsf}
\newtheorem{thm}{Theorem}
\newtheorem*{Cor}{Corollary}
\newtheorem{lem}{Lemma}
\newtheorem*{Lem}{Lemma}
\theoremstyle{definition}

\def\F{{\mathcal F}}
\def\S{{\mathcal S}}
\def\HH{{\mathbb H}}
\def\real{{\mathbb R}}
\def\Bone{{\rm 1~\hskip-1.4ex l} }
\begin{document}
\title[Functionals for Multilinear Fractional Embedding]
{Functionals for Multilinear Fractional Embedding}
\author{William Beckner}
\address{Department of Mathematics, The University of Texas at Austin,
1 University Station C1200, Austin TX 78712-0257 USA}
\email{beckner@math.utexas.edu}
\begin{abstract}
A novel representation is developed as a measure for multilinear fractional embedding. 
Corresponding extensions are given for the Bourgain-Brezis-Mironescu theorem 
and Pitt's inequality.
New results are obtained for diagonal trace restriction on submanifolds as an application 
of the Hardy-Littlewood-Sobolev inequality.
Smoothing estimates are used to provide new structural understanding for
density functional theory, the Coulomb interaction energy and
quantum mechanics of phase space. 
Intriguing connections are drawn that illustrate interplay among classical inequalities  in Fourier analysis.
\end{abstract}
\maketitle

\section{Multilinear embedding} 

A problem of central interest for embedding is how to characterize the action of multilinear 
fractional smoothing: that is, control by the operator 
\begin{equation}\label{eq:control}
\begin{split}
(\Lambda_\alpha f) (x_1,\ldots, x_m) 
& = \bigg[ \prod_{k=1}^m (-\Delta_k/4\pi^2)^{\alpha_k/2} f\bigg] (x_1,\ldots, x_m) \\
& = \F^{-1} \bigg[ \prod_k |\xi_k|^{\alpha_k}\widehat f (\xi_1,\ldots, \xi_m)\bigg]
(x_1,\ldots, x_m)
\end{split}
\end{equation}
where $f\in \S (\real^{mn})$, $x_k \in \real^n$, $\alpha = (\alpha_1,\ldots, \alpha_m)$, 
$|\alpha| = \sum \alpha_k$, $0< \alpha_k < n$, $\Delta_k$ is the standard Laplacian on 
$\real^n$ in the variable $x_k$, and 
$$(\F f) (\xi) = \widehat f ( \xi) = \int e^{2\pi ix\xi} f(x)\, dx\ .$$
Examples of how such control can be utilized are contained in (\cite{Beckner-MRL}, 
\cite{Beckner-2012}, \cite{Beckner-PAMS2013}, \cite{CS}).
Our objective here is to consider a corresponding functional suggested by the 
Aronszajn-Smith formula: 
\begin{equation}\label{eq:AS}
I_{p,\alpha} (f) 
= \int_{\real^{mn} \times \real^{mn}} 
\prod | x_k - y_k |^{-n-p\alpha_k}
\Big| \sum_{w=x,y} (-1)^{\sigma (y)} f(w_1,\ldots, w_m)\Big|^p\, dx\,dy
\end{equation}
where $f\in \S(\real^{mn})$, $\alpha = (\alpha_1,\ldots, \alpha_m)$, $0<\alpha_k < 1$, 
$1< p < n/\alpha_k$ for all $k$, and $\sigma (y)$ counts the number of $y$ values in the 
expression $f(w_1,w_2,\ldots, w_m)$ --- for example, $\sigma (y) =3$ in the case 
$f(x_1,y_2, y_3, x_4, y_5, x_6,\ldots, x_m)$. 
Related to this functional, one can give a non-local representation for multilinear fractional 
smoothness:
\begin{equation}\label{eq:non-local rep}
(\Lambda_\alpha f) (x_1,\ldots, x_m) 
= \prod_k (2/D_{\alpha_k}) \int_{\real^{mn}} \prod |x_k - y_k|^{-n-\alpha_k} 
\bigg[ \sum_{w=x,y} (-1)^{\sigma (y)} f (w_1,\ldots, w_m)\bigg] \,dy
\end{equation} 
Using the classical formula of Aronszajn-Smith and simple iteration: 

\begin{lem}[Multilinear Aronszajn-Smith Formula]
\begin{equation}\label{eq:AS-lem}
I_{2,\alpha} (f) = \Big[ \prod D_{2\alpha_k}\Big] \int_{\real^{mn}} |\Lambda_\alpha f|^2\,dx
\end{equation}
where 
$$D_\beta = \frac4{\beta} \ \pi^{-n/2 + \beta} \ 
\frac{\Gamma (1-\beta/2)}
{\Gamma (\frac{n+\beta}2)}
$$
and $0< \alpha_k < \min \{ 1,\, n/2\}$.
\end{lem}

\begin{proof}
Apply the classical Aronszajn-Smith formula to successive variables:
\begin{equation*}
\begin{split}
&\int |\widehat f (\xi_1,\xi_2,\xi')|^2 |\xi_1|^{2\alpha_1} |\xi_2|^{2\alpha_2} \, d\xi_1\, d\xi_2\\
&  = ( D_{2\alpha_1})^{-1} \int 
\frac{|f(x_1,\xi_2,\xi') - f(y_1,\xi_2,\xi')|^2}{|x_1 - y_1|^{n+2\alpha_1}} 
|\xi_2|^{2\alpha_2} , dx_1 \, dy_1\, d\xi_2\\
&  = (D_{2\alpha_1} D_{2\alpha_2})^{-1} \int 
\frac{|[ f(x_1, x_2,\xi') - f(y_1, x_2,\xi')] - [f(x_1,y_2,\xi') - f(y_1, y_2, \xi')]|^2}
{|x_1 - y_1|^{n+2\alpha_1} |x_2- y_2|^{n+2\alpha_2}}\ 
dx_1\, dx_2\, dy_1\, dy_2
\end{split}
\end{equation*}
where $\xi' = (\xi_3, \ldots, \xi_m)$. 
Continue this process until all the Fourier transform variables $\xi_k$ are exhausted.
\end{proof}

Observe that if $f(x_1,\ldots,x_m) = g(x_1)h (x_2,\ldots, x_m)$ then 
$$\sum_{w=x,y} (-1)^{\sigma (y)} f(w_1,\ldots, w_m) 
= \big[ g(x_1) - g(y_1)\big] 
\sum_{w=x,y} (-1)^{\sigma (y)} h (w_2,\ldots, w_m)$$
For product functions $f(x_1,\ldots, x_m) = \prod f_k (x_k)$
$$I_{p,\alpha} (f) = 
\prod_k \bigg[ \int_{\real^n\times\real^n} |x-y|^{-n-p\alpha_k} 
\big| f_k (x) - f_k (y)\big|^p \,dx\, dy\bigg]$$
This splitting, the utilization of iteration methods, and the product structure, 
suggests that the issue here is not a true multilinear problem, and that product 
functions will likely characterize results. 

\begin{thm}[Multilinear Pitt's Inequality] 
Let $f\in \S(\real^{mn})$, $0< \alpha_k < 1$ and $1\le p< \min \{n/\alpha_k\}$; then 
\begin{equation}\label{eq:MP}
I_{p,\alpha} (f)  \ge \prod_k D_{p,\alpha_k}\int_{\real^{mn}} \prod |x_k|^{-p\alpha_k} 
|f(x)|^p\, dx
\end{equation}
$$D_{p,\beta} = \int_{\real^n} \big| 1-|x|^{-\lambda}\big|^p |x-\eta|^{-n-p\beta}\, dx$$
for $\lambda = (n-p\beta) /p$ and $\eta \in S^{n-1}$.
\end{thm}

\begin{proof}
This result follows from successive application of Theorem~4.1 in \cite{Beckner-Forum} 
(see also Lemma~1 in \cite{Beckner-2012}). 
Observe that 
$$I_{p,\alpha} (f) \ge D_{p,\alpha_1} \int |x_1|^{-p\alpha_1} \prod_{k\ge 2} 
|x_k - y_k|^{-n-p\alpha_k} 
\bigg| \sum_{w=x,y} (-1)^{\sigma (y)} f(x_1,w_2,\ldots, w_m) \bigg|^p\, dx\, dy\ ;$$
continue this argument for the variables $x_k, y_k$ for $k\ge 2$ to obtain the full 
inequality \eqref{eq:MP}.
The constant is sharp as can be seen from the calculation for product functions.
\end{proof}

For $p=2$, a more explicit realization can be given for the constant (see discussion of 
Pitt's inequality in \cite{Beckner-PAMS08}, \cite{Beckner-Forum}):

\begin{Cor}
For $f\in \S (\real^{mn})$, $0< \alpha_k < 1$
\begin{gather} 
I_{2,\alpha} (f)   \ge \prod_k \Big[ D_{2\alpha_k}/C_{2\alpha_k}\Big] 
\int_{\real^{mn}} \prod |x_k|^{-2\alpha_k} |f(x)|^2\,dx\label{eq:Cor}\\
C_\beta = \pi^\beta  \left[ \Gamma \Big( \frac{n-\beta}4\Big) \Big/ \Gamma \Big(\frac{n+\beta}4
\Big)\right]^2 \nonumber
\end{gather}
\end{Cor}

\begin{thm}[Multilinear Bourgain-Brezis-Mironescu]
For $f\in \S(\real^{mn})$, $0<\beta <1$, $1\le p < n/\beta$ and 
$\alpha_k = \beta$ for all $k$; 
\begin{equation}\label{eq:MBBM}
I_{p,\alpha} (f) \ge (c_{n,p})^m 
\bigg( \int_{\real^{mn}} |f|^{q^*}\,dx \bigg)^{p/q^*}\ ,\qquad 
q^* = \frac{pn}{n-p\beta} 
\end{equation}
where $c_{n,p}$ is the optimal embedding constant on $\real^n$.
\end{thm}

\begin{Cor}
For $p = 2< n/\beta$, the value of $c_{n,2}$ is given by 
\begin{equation}\label{eq:optimal embed}
c_{n,2}  =\frac2{\beta}\ \frac{\Gamma (1-\beta)}{\Gamma (\frac{n}2 - \beta)} 
\left[ \frac{\Gamma (\frac{n}2)} {\Gamma (n)}\right]^{2\beta/n}
\end{equation}
\end{Cor}

\begin{proof}
This multilinear embedding result is obtained by applying the Bourgain-Brezis-Mironescu 
theorem in the context of a multiplicative iteration scheme with the aid of the 
Minkowski inequality for integrals:
\begin{equation*}
\begin{split}
&\int_{\real^{mn}\times \real^{mn}} \prod  |x_k - y_k|^{-n-p\alpha} 
\Big| \sum_{w=x,y} (-1)^\sigma f(w)\Big|^p\,dx\,dy\\
\noalign{\vskip6pt}
&\quad \ge 
c_{n,p} \int_{\real^{(m-1)n} \times \real^{(m-1)n}} 
\prod |x'_k - y'_k|^{-n-p\alpha}
\bigg[ \int_{\real^n} \Big| \sum_{w=x',y'} (-1)^\sigma f(v_1,w)\Big|^{q^*} dv_1\bigg]^{p/q^*}
dx'\,dy'\\
\noalign{\vskip6pt}
&\quad \ge c_{n,p} \bigg[ \int_{\real^n} \bigg[ \int_{\real^{(m-1)n}\times \real^{(m-1)n}} 
\prod |x'_k - y'_k|^{-n-p\alpha} 
\Big| \sum_{w=x',y'} (-1)^\sigma f(v_1,w)\Big|^p dx'\,dy'\bigg]^{q^*/p} dv_1\bigg]^{p/q^*}\\
\noalign{\vskip6pt}
&\quad \ge 
(c_{n,p})^2 \bigg[ \int_{\real^n} \bigg[ \int_{\real^{(m-2)n}\times\real^{(m-2)n}} 
\prod |x''_k - y''_k|^{-n-p\alpha}\\
\noalign{\vskip6pt}
&\hskip1.5truein 
\bigg[ \int_{\real^n} \Big| \sum_{w=x'',y''} (-1)^\sigma f(v_1,v_2,w)\Big|^{q^*} dv_2\bigg]^{p/q^*}
dx''\, dy''\bigg]^{q^*/p} dv_1\bigg]^{p/q^*}\\
\noalign{\vskip6pt}
&\qquad \ge (c_{n,p})^2 \bigg[ \int_{\real^n\times \real^n} 
\bigg[ \int_{\real^{(m-2)n} \times\real^{(m-2)n}} 
\prod |x''_k - y''_k |^{-n-p\alpha}\\
\noalign{\vskip6pt}
&\hskip1.5truein 
\Big| \sum_{w=x'', y''} (-1)^\sigma f(v_1,v_2,w)\Big|^p dx''\, dy''\bigg]^{q^*/p} 
dv_1\, dv_2\bigg]^{p/q^*} \\
\noalign{\vskip6pt}
&\qquad \ge \cdots = (c_{n,p})^m 
\bigg[ \int_{\real^{mn}} |f(v)|^{q^*} dv\bigg]^{p/q^*}
\end{split}
\end{equation*}
Here primes denote: $x' = (x_2,\ldots ,x_m)$ and $x'' = (x_3,\ldots, x_m)$. 
The first inequality follows from application of the Bourgain-Brezis-Mironescu theorem on $\real^n$; 
the second inequality invokes Minkowski's inequality for integrals in the form 
$$\int \bigg[ \int |h|^q\,d\mu\bigg]^{p/q} \,d\nu 
\ge \bigg[ \int \bigg[ \int |h|^p\,d\nu\bigg]^{q/p}\,d\mu\bigg]^{p/q}\ ,\qquad q>p\ .$$
The sharpness of the constant is demonstrated 
by using product functions --- $f(x) = \prod f_k (x_k)$.
The sharp $L^2$ embedding constant $c_{n,2}$ was first noted in \cite{Beckner-Forum} 
(see Theorem~3.3 on page 187).
\end{proof}

\section{Diagonal trace restriction} 

The objective here is to develop an overall framework for the structure of multilinear 
convolution operators and the representation of the Hardy-Littlewood-Sobolev inequality 
from the perspective defined by multilinear Sobolev embedding. 
To enable a better understanding for the role of geometric symmetry and the application 
of duality arguments, diagonal trace restriction is considered in the context of a 
lower-dimensional manifold --- namely, the unit sphere. 
This approach extends the structure of classical trace inequalities from harmonic extension 
of boundary values (for the upper half-space or the interior of the unit ball, see 
\cite{Beckner-Annals93}, \cite{E}) to restriction 
phenomena on surfaces with curvature.
Questions about restriction for the Fourier transform on manifolds with curvature and 
Strichartz inequalities involve greater depth and subtlety as illustrated by the original 
Stein-Tomas inequality.

Determination of sharp constants for diagonal trace restriction estimates was initiated 
in \cite{Beckner-MRL} and extended in \cite{Beckner-PAMS2013}.
Motivated by the principal results from \cite{Beckner-MRL}  (Theorems~1 and 2), 
restriction estimates are obtained here for the sphere $S^{n-1}$.
First consider the basic estimates 

\subsection*{Pitt's inequality} ($n-\beta = mn - 2\alpha$, $\alpha = \sum \alpha_k$, $0<\beta<n$)
\begin{gather}
\int_{\real^n} |x|^{-\beta} |f(\underbrace{x,\cdots,x}_{m \text{ slots}}) |^2\,dx 
\le C_\beta \int_{\real^n\times\cdots \times\real^n} 
|\Lambda_\alpha f|^2\,dx \label{eq:Pitt}\\
\noalign{\vskip6pt}
C_\beta = \pi^{-(m-1)n/2 \, + \, 2\alpha} 
\prod_{k=1}^m \left[\frac{\Gamma (\frac{n}2 - \alpha_k)}{\Gamma (\alpha_k)}\right]
\left[ \frac{\Gamma (\frac{\beta}2)}{\Gamma (\frac{n-\beta}2)}\right] 
\left[ \frac{\Gamma (\frac{n-\beta}4)}{\Gamma (\frac{n+\beta}4)}\right]^2 
\nonumber
\end{gather}

\subsection*{Hardy-Littlewood-Sobolev inequality}  $(mn-2\alpha = 2n/q)$
\begin{gather} 
\bigg[ \int_{\real^n}  |f(\underbrace{x,\cdots,x}_{m \text{ slots}})|^q \,dx\bigg]^{2/q}
\le F_\alpha \int_{\real^n \times\cdots\times \real^n} |\Lambda_\alpha f|^2\,dx\label{eq:HLS}\\
\noalign{\vskip6pt}
F_\alpha = \pi^\alpha \prod_{k=1}^m 
\left[ \frac{\Gamma (\frac{n}2 - \alpha_k)}{\Gamma (\alpha_k)}\right]
\left[ \frac{\Gamma (\alpha - (m-1)n/2)}{\Gamma (\alpha - (m-2)n/2)}\right]
\left[ \frac{\Gamma (n)}{\Gamma (n/2)}\right]^{\frac{2\alpha - (m-1)n}n}
\nonumber
\end{gather}

Motivated by the proof of the Hardy-Littlewood-Sobolev inequality in \cite{Beckner-MRL}, 
a diagonal trace restriction inequality can be given in terms of the $(n-1)$ dimensional 
unit sphere. 
The proof uses duality and a reduction to the $(n-1)$ dimensional Hardy-Littlewood-Sobolev
inequality on the sphere. 

\begin{thm}[Multilinear Hardy-Littlewood-Sobolev] \label{thm3}
For $f\in \S(\real^{mn})$ and $mn-2\alpha = {2(n-1)/q}$, $q>2$, $n> 1$
\begin{gather}
\bigg[ \int_{\S^{n-1}}  |f(\underbrace{\xi,\cdots,\xi}_{m \text{ slots}})|^q \,d\xi \bigg]^{2/q}
\le K_\alpha \int_{\real^n \times\cdots\times \real^n} |\Lambda_\alpha f|^2\,dx\label{eq:HLS2}\\
\noalign{\vskip6pt} 
K_\alpha = (2\pi)^{2\alpha} (4\pi)^{-mn/2}  \prod_{k=1}^m 
\  \frac{\Gamma (\frac{n}2 - \alpha_k)}{\Gamma (\alpha_k)} 
\   \frac{\Gamma (n-1)}{\Gamma (\frac{n-1}2)} 
\   \frac{\Gamma [(n-1)(\frac12 -\frac1q)]}{\Gamma (\frac{n-1}p)}  
\nonumber
\end{gather}
Here $d\xi$ denotes normalized surface measure on the sphere $S^{n-1}$.
\end{thm}

\begin{proof}
Inequality \eqref{eq:HLS2} is equivalent to the multilinear fractional integral inequality: 
\begin{gather}
\bigg[ \int_{S^{n-1}} \Big| \int_{\real^{mn}} \prod_{k=1}^m |\xi - y_k|^{-(n-\alpha_k)} 
f(y_1,\ldots, y_m)\,dy\Big|^q\,d\xi \bigg]^{2/q}\nonumber\\
\noalign{\vskip6pt}
\le G_\alpha \int_{\real^{mn}} |f(x_1,\ldots, x_m)|^2\,dx \label{eq:HLS-pf}\\
\noalign{\vskip6pt}
K_\alpha = \pi^{-mn+2\alpha} \prod_{k=1}^m \left[ \frac{\Gamma (\frac{n-\alpha_k}2)}
{\Gamma (\frac{\alpha_k}2)}\right]^2\, G_\alpha \nonumber
\end{gather}
By duality this is equivalent to 
\begin{equation*}
\int_{\real^{mn}} \Big| \int_{S^{n-1}} \prod_{k=1}^m |y_k-\xi|^{-(n-\alpha_k)} g(\xi)\,d\xi \Big|^2\,dy
\le G_\alpha \bigg[ \int_{S^{n-1}} |g(\xi)|^p\,d\xi \bigg]^{2/p}
\end{equation*}
where $1/p + 1/q =1$, $1<p<2$ and $mn - 2\alpha = 2(n-1)/q$.
The left-hand side now becomes
\begin{equation*}
\int_{S^{n-1}\times S^{n-1}\times\real^{mn}} g(\xi) \prod_{k=1}^m |y_k -\xi|^{-(n-\alpha_k)} 
\prod_{k=1}^m |y_k -\eta|^{-(n-\alpha_k)} g(\eta)\, d\xi\,d\eta\, dy
\end{equation*}
Integrating out the $y_k$ variables
\begin{gather*}
\int_{S^{n-1} \times S^{n-1}} g(\xi) |\xi-\eta|^{-mn+2\alpha} 
g(\eta)\, d\xi\,d \eta 
\le H_\alpha \bigg[ \int_{S^{n-1}} |g(\xi)|^p \,d\xi\bigg]^{2/p}\\
\noalign{\vskip6pt}
K_\alpha = \pi^{-mn/2 + 2\alpha} \prod_{k=1}^m \Gamma \Big( \frac{n}2 - \alpha_k\Big) 
\Big/ \Gamma (\alpha_k)\ H_\alpha
\end{gather*}
Since $mn- 2\alpha = 2(n-1)/q$, this becomes the classical Hardy-Littlewood-Sobolev 
inequality on the $(n-1)$ dimensional sphere $S^{n-1}$:
\begin{gather*}
\int_{S^{n-1}\times S^{n-1}} g(\xi) |\xi-\eta|^{-2(n-1)/q} g(\eta)\, d\xi\, d\eta 
\le H_\alpha \bigg[ \int_{S^{n-1}} |g(\xi)|^p\,d\xi\bigg]^{2/p} \\
\noalign{\vskip6pt}
H_\alpha = 2^{-2(n-1)/q}\ \frac{\Gamma (n-1)}{\Gamma (\frac{n-1}2)}
\ \frac{\Gamma [(n-1)(\frac12 - \frac1q)]}{\Gamma (\frac{n-1}p)}
\end{gather*}
Then
$$K_\alpha = (2\pi)^{2\alpha} (4\pi)^{-mn/2} \prod_{k=1}^m\ 
\frac{\Gamma (\frac{n}2 - \alpha_k)}{\Gamma (\alpha_k)} \ 
\frac{\Gamma (n-1)}{\Gamma (\frac{n-1}2)}\ 
\frac{\Gamma [(n-1)(\frac12 - \frac1q)]}{\Gamma (\frac{n-1}p)}\ .
$$
Extremal functions for \eqref{eq:HLS2} and \eqref{eq:HLS-pf} are determined up to 
conformal automorphism on the sphere $S^{n-1}$ as equivalent to 
$$\int_{S^{n-1}} \prod_{k=1}^m |x_k - \xi|^{-(n-\alpha_k)} d\xi\ .$$
\end{proof}

Observe that the duality argument used in this proof provides the following restriction 
result for a spherical surface as determined by fractional smoothness.
This result was obtained earlier by Bez, Machihara and Sugimoto (personal communication
--- see \cite{BMS}). 
From the representation of the Hardy-Littlewood-Sobolev inequality as a smoothing 
estimate, one expects similar estimates to hold for any conformally equivalent setting 
(see equation~\eqref{eq:HLS} above and section~2 in \cite{Beckner-Annals93}).

\begin{thm}\label{thm4} 
For $f\in \S (\real^n)$ and $n-2\alpha = 2(n-1)/q$ with $q>2$, $n>1$
\begin{gather} 
\bigg[ \int_{S^{n-1}} |f(\xi)|^q\,d\xi\bigg]^{2/q} 
\le B_\alpha \int_{\real^n} \Big| (-\Delta/4\pi^2)^{\alpha/2} f\Big|^2\,dx \label{eq:thm4}\\
\noalign{\vskip6pt}
B_\alpha = (2\pi)^{2\alpha} (4\pi)^{-n/2} 
\frac{\Gamma (\frac{n-1}q)}{\Gamma (\frac{n-1}p)} \ 
\frac{\Gamma (n-1)}{\Gamma (\frac{n-1}2)}\ 
\frac{\Gamma [(n-1)(\frac12-\frac1q)]}{\Gamma [(n-1)(\frac12-\frac1q) +\frac12]}\nonumber
\end{gather}
\end{thm}

\begin{proof} 
This inequality corresponds to the special case $m=1$ in the previous argument and demonstrates
that such estimates additionally hold for spherical restriction for all positive indices below 
the critical index $q = 2(n-1)/(n-2\alpha)$.
\end{proof}

To gain a better sense of the contrast for this trace estimate between harmonic extentsion 
of boundary values and global embedding, set $\alpha=1$ and raise the dimension by one 
so that critical index is given by $q= 2n/(n-1)$ for $n>1$; then 
\begin{gather} 
\bigg( \int_{S^n} |f(\xi)|^q\,d\xi\bigg)^{2/q} 
\le b_n \int_{\real^{n+1}} |\nabla f|^2\, dx \label{eq:thm4-pf}\\
\noalign{\vskip6pt}
b_n = \frac14\ \pi^{-(n+1)/2}\ \Gamma \Big( \frac{n-1}2\Big)\ .\nonumber
\end{gather}

On the other hand, using Theorem~4 from \cite{Beckner-Annals93} 
for the critical index $q= 2n/(n-1)$ 
\begin{equation}\label{eq:BA93}
\bigg( \int_{S^n} |F(\xi)|^q\,d\xi\bigg)^{2/q} 
\le 2b_n \int_{|x|\le 1} |\nabla u|^2\,dx + \int_{S^n} |F(\xi)|^2\, d\xi
\end{equation}
where $u$ is the harmonic extension of $F$ to the interior of the unit ball in $\real^{n+1}$.
Both inequalities are sharp, and the doubling factor seems natural in view of symmetry. 
A precise derivation of the relation between the two inequalities using symmetrization is given
in the Appendix.

The possibility of considering a spherical trace diagonal restriction corresponding to Pitt's 
inequality is less natural because the estimate is taken over a compact domain, and 
the critical index for embedding is not used. 
The nature of Pitt's inequality depends on the dilation character of the smoothing operator 
which will not play a new role for restriction on a compact manifold. 
Moreover, in contrast to the non-compact setting where extremals do not exist, one 
expects that in the compact case extremals are likely to exist.

\section{Diagonal trace restriction on submanifolds}

Trace restriction from either the vantage point of harmonic extension or understanding 
models for many-body dynamics using the Gross-Pitaevskii hierarchy of density matrices 
seems naturally associated with control determined by multilinear fractional Sobolev embedding.
But diagonal trace restriction on submanifolds is more directly a consequence of the 
Hardy-Littlewood-Sobolev embedding estimates, including more general formulations.
First, a very general principle is outlined, and then explicit applications are developed 
including the case of flat submanifolds.

\subsection*{Hardy-Littlewood-Sobolev principle --- submanifold restriction} 
For $f\in \S (\real^{mn})$, $K$ a smooth submanifold of $\real^n$, $\sigma$ denotes a 
surface measure on $K$, and the index $q$ depends on $\alpha$ and $K$; then 
\begin{equation}\label{eq:HLS principle}
\bigg[ \int_{K\times \cdots \times K} |f(\,\underbrace{w,\ldots,w}_{m\text{ slots}}\, )|^q\, 
d\sigma\bigg]^{2/q}
\le C_\alpha \int_{\real^n\times\cdots\times \real^n} |\Lambda_\alpha f|^2\,dx
\end{equation}
This result is determined by the corresponding Hardy-Littlewood-Sobolev inequality on $K$:
\begin{equation}\label{eq:HLS K inequal}
\Big| \int_{K\times K} g(u) |u-v|^{-\lambda} h(v)\,d\sigma\,d\sigma\Big| 
\le C'_\alpha \|g\|_{L^p (K)} \|h\|_{L^p (K)} 
\end{equation}
where $p$ is the dual exponent to $q$ and $\lambda = mn - 2\alpha$, 
$\alpha = \sum \alpha_k$.

The classical sense of trace operator is associated with harmonic extension and solutions 
of differential equations.
But here consideration of diagonal trace restriction suggests a broader mechanism that 
couples fractional Sobolev embedding with estimates for multilinear potential operators 
and application of the Hardy-Littlewood-Sobolev inequality to obtain optimal bounds.
Without being exhaustive, examples are given to suggest the range of results that may 
be obtained using diagonal trace restriction on submanifolds, including both flat and 
product submanifolds. 

\begin{thm}[flat submanifolds]
\label{thm:flat}
For $f\in \S (\real^{mn})$, $n = k+\ell$, $1\le k,\ell$ and $\bar x = (x,y)$ for $x\in \real^k$ and
$y$ a fixed point in $\real^\ell$ with $mn-2\alpha = 2k/q$, $q>2$:
\begin{gather}
\bigg[ \int_{\real^k} |f(\, \underbrace{\bar x,\ldots, \bar x}_{m\text{ slots}}\, )|^q\,dx \bigg]^{2/q} 
\le A_\alpha \int_{\real^{mn}} |\Lambda_\alpha f|^2 dx_1\ldots dx_m \label{eq:flat1}\\
\noalign{\vskip6pt}
A_\alpha = 
\pi^\alpha \ \frac{\Gamma [k(\frac1p - \frac12)]}{\Gamma (\frac{k}p)} \ 
\left[ \frac{\Gamma (\frac{k}2)}{\Gamma (k)}\right]^{1-2/p} \prod_{\ell=1}^m 
\frac{\Gamma (\frac{n}2 -\alpha_\ell)}{\Gamma (\alpha_\ell)}\ .
\nonumber
\end{gather}
\end{thm}

\begin{thm}[Pitt's inequality]
\label{thm:Pitt}
For $f\in \S (\real^{mn})$, $n= k+\ell$, $1\le k,\ell$ and $\bar x = (x,y)$ for $x\in \real^k$ and $y$ 
a fixed point in $\real^\ell$ with $mn-2\alpha = k-\beta$, $0<\beta < k$:
\begin{gather} 
\int_{\real^k} |x|^{-\beta} |f(\, \underbrace{\bar x,\ldots,\bar x}_{m\text{ slots}}\, )|^2 \,dx 
\le C_{\beta,k} \int_{\real^n\times \cdots \times \real^n} 
|\Lambda_\alpha f|^2\, dx_1 \ldots dx_m \label{eq:Pitts thm}\\
\noalign{\vskip6pt}
C_{\beta,k} =  
 \pi^{(-mn+k)/2 + 2\alpha}\ \prod_{j=1}^m 
\frac{\Gamma (\frac{n}2 - \alpha_j)}{\Gamma (\alpha_j)} 
\left[ \frac{\Gamma (\frac{\beta}2)}{\Gamma (\frac{k-\beta}2)}\right]
\left[ \frac{\Gamma (\frac{k-\beta}4)}{\Gamma (\frac{k+\beta}4)}\right]^2 
\nonumber
\end{gather}
\end{thm}

\begin{thm}[product of spheres]
\label{thm:product}
For $f\in \S(\real^{mn})$, $n= k+\ell + 2$, $1\le k,\ell$ and $\bar\xi = (\xi,\eta)$, 
$\xi \in S^k$, $\eta\in \S^\ell$ with $d\xi \, d\eta$ denoting normalized surface measure 
on $S^k \times S^\ell$ with $mn - 2\alpha = 2(k+\ell)/q$, $q>2$:
\begin{gather} 
\bigg[ \int_{S^k\times S^\ell} |f(\, \underbrace{\bar \xi ,\ldots, \bar \xi}_{m\text{ slots}}\,)|^q\,d\xi \, d\eta\bigg]^{2/q} 
\le B_{\alpha,k} \int_{\real^{mn}} |\Lambda_\alpha f|^2\,dx \label{eq:product}\\
\noalign{\vskip6pt}
B_{\alpha,k} =  
(4\pi)^{-(k+\ell)/q}\  \pi^\alpha\prod_{j=1}^m 
\frac{\Gamma (\frac{n}2 - \alpha_j)}{\Gamma (\alpha_j)}\ 
\frac{\Gamma (k) \Gamma (\ell)}{\Gamma (k/2)\Gamma (\ell/2)}\ 
\frac{\Gamma [k(\frac12-\frac1q)] \Gamma (\ell (\frac12 -\frac1q)]}{\Gamma (k/p)\Gamma (\ell/p)} 
\nonumber
\end{gather}
\end{thm}

\begin{proof}[Proof of Theorem~\ref{thm:flat}]
Inequality \eqref{eq:flat1} is equivalent to the multilinear fractional integral inequality: 
\begin{gather*}
\bigg[ \int_{\real^k} \Big| \int_{\real^{mn}} \prod_{\ell=1}^m |\bar x - u_\ell|^{-(n-\alpha_\ell)} 
f(u_1,\ldots, u_m)\,du\Big|^q\,dx \bigg]^{2/q}\\ 				
\noalign{\vskip6pt}
\le A_{\alpha,1} \int_{\real^{mn}} |f(u_1,\ldots, u_m)|^2\,du\\ 	
\noalign{\vskip6pt}
A_\alpha = \pi^{-mn+2\alpha} \prod_{\ell =1}^m \left[ \frac{\Gamma (\frac{n-\alpha_\ell }2)}
{\Gamma (\frac{\alpha_\ell }2)}\right]^2\, A_{\alpha,1}  		
\end{gather*}
By duality this is equivalent to 
\begin{equation*}
\int_{\real^{mn}} \Big| \int_{\real^k} \prod_{\ell =1}^m |u_\ell -\bar x|^{-(n-\alpha_\ell)} g(x)\,dx 
\Big|^2\,du
\le A_{\alpha,1} \bigg[ \int_{\real^k} |g(x)|^p\,dx \bigg]^{2/p}
\end{equation*}
where $1/p + 1/q =1$, $1<p<2$ and $mn - 2\alpha = 2k/q$.
The left-hand side now becomes
\begin{equation*}
\int_{\real^k \times \real^k \times \real^{mn}}   g(x) \prod_{\ell =1}^m |u_\ell -\bar x |^{-(n-\alpha_\ell)} 
\prod_{\ell =1}^m |u_\ell -\bar w |^{-(n-\alpha_\ell)} g(w)\, dx\, dw\, du
\end{equation*}
where $\bar x = (x,y)$   $\bar w = (w,y)$ with $y$ is a fixed point in $\real^{n-k}$. 
But now   this form is independent of the fixed point $y$. 
Integrating out the $u_k$ variables 
\begin{gather*}
\int_{\real^k \times \real^k } 
 g(x) |x-w|^{-mn+2\alpha} 
g(w)\, dx\, dw
\le A_{\alpha,2}  \bigg[ \int_{\real^k} |g(x|^p \,dx \bigg]^{2/p}\\
\noalign{\vskip6pt}
A_\alpha = \pi^{-mn/2 + 2\alpha} \prod_{\ell =1}^m \Gamma \Big( \frac{n}2 - \alpha_\ell\Big) 
\Big/ \Gamma (\alpha_\ell)\  A_{\alpha,2}
\end{gather*}
Since $mn- 2\alpha = 2k/q$, this estimate becomes the classical Hardy-Littlewood-Sobolev 
inequality on the $\real_k$:
\begin{gather*}
\int_{\real^k \times \real^k }   g(w) |x-w|^{-2k/q} g(w)\, dx\, dw  
\le A_{\alpha,2} \bigg[ \int_{\real^k} |g(x)|^p\,dx \bigg]^{2/p}\\
\noalign{\vskip6pt}
A_{\alpha,2} = \pi^{k/q} \ \frac{ \Gamma [k (\frac1p - \frac12)}{\Gamma  (\frac{k}p)}   
\left[ \frac{\Gamma (\frac{k}2)}{\Gamma (k)}\right]^{1- 2/p}
\end{gather*}
Then
\begin{equation*}
A_\alpha = \pi^\alpha \ \frac{\Gamma [k(\frac1p - \frac12)]}{\Gamma (\frac{k}p)} \ 
\left[ \frac{\Gamma (\frac{k}2)}{\Gamma (k)}\right]^{1-2/p} \prod_{\ell=1}^m 
\frac{\Gamma (\frac{n}2 -\alpha_\ell)}{\Gamma (\alpha_\ell)}\ .
\end{equation*}
\end{proof}

\begin{proof}[Proof of Theorem~\ref{thm:Pitt}]
Inequality \eqref{eq:Pitts thm} is equivalent to the multilinear fractional integral inequality: 
\begin{gather*}
\int_{\real^k} |x|^{-\beta}\ \Big| \int_{\real^{mn}} \prod_{j=1}^m |\bar x- u_j|^{-(n-\alpha_j)} 
f(u_1,\ldots ,u_m) \,du\Big|^2\, dx\\
\noalign{\vskip6pt}
\le E_{\alpha,1} \int_{\real^{mn}} |f(u_1,\ldots, u_m)|^2\, dx\\
\noalign{\vskip6pt}
C_{\beta,k} = \pi^{-mn + 2\alpha}\ \prod_{j=1}^m 
\left[ \frac{\Gamma (\frac{n-\alpha_j}2)} {\Gamma (\frac{\alpha_j}2)}\right]^2\ E_{\alpha_1}
\end{gather*}
By duality this is equivalent to 
\begin{equation*}
\int_{\real^{mn}} \Big| \int_{\real^k} \prod_{j=1}^m |u_j -\bar x|^{-(n-\alpha_k)} 
|x|^{-\beta/2} g(x)\,dx \Big|^2\, du 
\le E_{\alpha,1} \int_{\real^k}|g|^2\, dx
\end{equation*}
where $mn- 2\alpha = k-\beta$.
The left-hand side now becomes
\begin{equation*} 
\int_{\real^k \times\real^k \times \real^{mn}} g(x) |x|^{-\beta/2} 
\prod_{j=1}^m |u_j - \bar x|^{-(n-\alpha_j)} 
\prod_{j=1}^m |u_j - \bar w|^{-(n-\alpha_j)} |w|^{-\beta/2} g(w)\, dx\,dw\,du
\end{equation*}
where $\bar x = (x,y)$, $\bar w = (w,y)$ with $y$ a fixed point in $\real^{n-k}$. 
But now this form is independent of the fixed point $y$.
Integrating out the $u_j$ variables 
\begin{gather*}
\int_{\real^k\times \real^k} g(x) |x|^{-\beta/2} |x-w|^{-mn+2\alpha} |w|^{-\beta/2} g(w)\, dx\, dw
\le E_{\alpha,2} \int_{\real^k} |g|^2\, dx\\
\noalign{\vskip6pt}
C_{\beta,k} = \pi^{-mn/2 + 2\alpha} 
\prod_{j=1}^m \Gamma \Big( \frac{n}2 - \alpha_j\Big)\big/ \Gamma (\alpha_j)\ E_{\alpha,2} 
\end{gather*}
Since $mn- 2\alpha = k-\beta$, this estimate becomes the classical Pitt's inequality on $\real^k$:
\begin{gather*}
\int_{\real^k \times\real^k} g(x) |x|^{-\beta/2} |x-w|^{-(k-\beta)} |w|^{-\beta/2} \,dx\, dw
\le E_{\alpha,2} \int_{\real^k} |g|^2\, dx\\
\noalign{\vskip6pt}
E_{\alpha,2} = \pi^{k/2}\ \left[\frac{\Gamma (\frac{\beta}2)}{\Gamma (\frac{k-\beta}2)}\right]
 \left[ \frac{\Gamma (\frac{k-\beta}4)}{\Gamma (\frac{k+\beta}4)}\right]^2
\end{gather*}
Then
\begin{equation*}
C_{\beta,k} = \pi^{(-mn+k)/2 + 2\alpha}\ \prod_{j=1}^m 
\frac{\Gamma (\frac{n}2 - \alpha_j)}{\Gamma (\alpha_j)} 
\left[ \frac{\Gamma (\frac{\beta}2)}{\Gamma (\frac{k-\beta}2)}\right]
\left[ \frac{\Gamma (\frac{k-\beta}4)}{\Gamma (\frac{k+\beta}4)}\right]^2
\end{equation*}
\end{proof}

\begin{proof}[Proof of Theorem~\ref{thm:product}]
Inequality \eqref{eq:product} is equivalent to the multilinear fractional   inequality
where $\hat x = (\xi,\eta) \in S^k \times S^\ell$:
\begin{gather*}
\int_{S^k \times S^\ell} \Big| \int_{\real^{mn}} \prod_{j=1}^m |\hat x - y_j|^{-(n-\alpha_k)} 
f(y_1,\ldots, y_m)\, dy\Big|^q\, d\xi\, d\eta\bigg]^{2/q}\\
\noalign{\vskip6pt}
\le F_{\alpha,1} \int_{\real^{mn}} |f(x_1,\ldots, x_m)|^2\, dx\\
\noalign{\vskip6pt}
B_{\alpha,k} = \pi^{-mn+2\alpha} \ \prod_{k=1}^m 
\left[ \frac{\Gamma (\frac{n-\alpha_k}2)}{\Gamma (\frac{\alpha_k}2)}\right]^2\ F_{\alpha,1}
\end{gather*}
By duality this is equivalent to 
\begin{equation*}
\int_{\real^{mn}} \Big| \int_{S^k\times S^\ell} \prod_{j=1}^m\ |u_j - \hat x|^{-(n-\alpha_j)} 
g(\hat x) \,d\xi\, d\eta\Big|^2\, du 
\le F_{\alpha,1} \bigg[ \int_{S^k\times S^\ell} |g(\hat x)|^p\,d\xi\, d\eta\bigg]^{2/p}
\end{equation*}
where $1/p + 1/q =1$, $1< p < 2$ and $mn-2\alpha = 2(k+\ell)/q$. 
The left-hand side now becomes
\begin{equation*}
\int_{M\times M\times \real^{mn}} g(\bar x) \prod_{j=1}^m |u_j - \bar x|^{-(n-\alpha_j)} 
\prod_{j=1}^m |u_j - \bar w|^{-(n-\alpha_j)} g(\bar w) \, d\bar x\, d\bar w\, du
\end{equation*}
where $M = S^k \times S^k$, $\bar x = (\xi,\eta)$, $\bar w = (\xi',\eta')$. 
Integrating out the $u_k$ variables 
\begin{gather*}
\int_{M\times M} g(\bar x) \Big[ |\xi -\xi'|^2 + |\eta - \eta'|^2\Big]^{-mn/2\ +\  \alpha} 
g(\bar w) \, d\bar x\, d\bar w 
\le F_{\alpha,2} \bigg[ \int_M |g|^p \, d\bar x\bigg]^{2/p}\\
\noalign{\vskip6pt}
B_{\alpha,k} = \pi^{-mn/2\ +\ 2\alpha} \ \prod_{j=1}^m \Gamma \Big( \frac{n}2 -\alpha/j\Big)
\Big/ \Gamma (\alpha_j)\ F_{\alpha,2}
\end{gather*}
Since $mn - 2\alpha = 2(k+\ell)/q$ and $[|\xi-\xi'|^2 + |\eta - \eta'|^2]^{k+\ell} \ge 
|\xi-\xi'|^{2k} |\eta - \eta'|^{2\ell}$
\begin{gather*}
\int_{M\times M}  g(\xi,\eta) \Big[ |\xi -\xi'|^2 + |\eta - \eta'|^2\Big]^{-(k+\ell)/q} 
g(\xi',\eta') \,d\xi\,d\eta\, d\xi' \, d\eta'\\
\noalign{\vskip6pt}
\le \int_{M\times M} g(\xi,\eta) |\xi-\xi'|^{2k/q} |\eta - \eta'|^{-2\ell/q} g(\xi',\eta')\, 
d\xi\, d\eta\, d\xi'\, d\eta'\\
\noalign{\vskip6pt} 
\le F_{\alpha,3} \bigg[ \int_M |g(\xi,\eta)|^p \,d\xi\,d\eta\bigg]^{2/p}\ .
\end{gather*}
By using successive applications of the Hardy-Littlewood-Sobolev inequality on spheres 
\begin{equation*}
F_{\alpha,3} = 2^{-(k+\ell)/q}\ 
\frac{\Gamma (k) \Gamma (\ell)}{\Gamma (k/2)\Gamma (\ell/2)} \ 
\frac{\Gamma [k(\frac12 - \frac1q)] \Gamma [\ell (\frac12 -\frac1q)]}
{\Gamma (k/p) \Gamma (\ell/p)}
\end{equation*}
Since $F_{\alpha,3} > F_{\alpha,2}$, a non-sharp value of $B_{\alpha,k}$ is given by:
\begin{equation*}
B_{\alpha,k} = (4\pi)^{-(k+\ell)/q}\ \pi^\alpha\ \prod_{j=1}^m\ 
\frac{\Gamma (\frac{n}2 - \alpha_j)}{\Gamma (\alpha_j)}\ 
\frac{\Gamma (k) \Gamma (\ell)}{\Gamma (k/2) \Gamma (\ell/2)}\ 
\frac{\Gamma [k(\frac12 - \frac1q)] \Gamma [\ell (\frac12 - \frac1q)]}
{\Gamma (k/p) \Gamma (\ell/p)}
\end{equation*}
\end{proof}

In contrast to Theorems~\ref{thm:flat} and \ref{thm:Pitt} where the constants are sharp, 
the resulting constant $B_{\alpha,k}$ obtained here for Theorem~\ref{thm:product} is not 
sharp because of using the geometric mean estimate. 
Further, embedding restriction for a product submanifold of spheres allows the embedding 
index~$q$ to decrease, that is to be closer to the index~2. 

As observed above for the sphere, Theorem~\ref{thm:flat} will determine a restriction result
for a subspace that includes the usual estimates for harmonic extension to a half-space.

\begin{thm}\label{thm:corresponds flat}
For $f\in \S(\real^n)$ and $\bar x = (x,0)$ with $x\in \real^{n-1}$ with $n-2\alpha = 2(n-1)/q$, 
$q>2$ $(2\alpha - 1 >0)$
\begin{gather} 
\bigg[ \int_{\real^{n-1}} |f(\bar x)|^q\, dx\bigg]^{2/q} 
\le A_\alpha \int_{\real^n} |\Lambda_\alpha f|^2\,dx \label{eq:corresponds}\\
\noalign{\vskip6pt}
A_\alpha = \pi^\alpha \ \frac{\Gamma (\alpha - \frac12)}{\Gamma (\frac{n}2 + \alpha-1)}\ 
\frac{\Gamma (\frac{n}2 - \alpha)}{\Gamma (\alpha)} 
\left[ \frac{\Gamma (n-1)}{\Gamma (\frac{n-1}2)}\right]^{\frac{2\alpha -1}{n-1}}
\nonumber
\end{gather}
\end{thm}

\begin{proof}
This estimate corresponds to the case $m=1$ in Theorem~\ref{thm:flat}.
\end{proof}

To match this result to harmonic extension, set $\alpha =1$ and raise the dimension by 
one so that the critical index is given by $q = 2n /(n-1)$; then for $(x,y)\in \real^{n+1}$
\begin{gather} 
\bigg( \int_{\real^n} |f(x,0)|^q\,dx\bigg)^{2/q} 
\le c_n \int_{\real^{n+1}} |\nabla f|^2\,dx\, dy \label{eq:match}\\
\noalign{\vskip6pt}
c_n = \frac1{\sqrt{4\pi}}\ \frac1{n-1} \left[ \frac{\Gamma (n)}{\Gamma (n/2)}\right]^{1/n}
\nonumber
\end{gather} 
But this inequality determines the classic result for harmonic extension on a half-space 
(see inequality~32 on page~231 in \cite{Beckner-Annals93}): 
\begin{equation}\label{eq:classic} 
\bigg( \int_{\real^n} |f(x)|^q\,dx\bigg)^{2/q} 
\le 2c_n \int_{\real_+^{n+1}} |\nabla u|^2\,dx\,dy 
\end{equation} 
where $n>1$ and $u$ is the harmonic extension of $f$ to the upper half-space. 

\section{Fractional embedding on the sphere}

The emergence of restriction smoothing estimates on the sphere suggests that embedding 
estimates for fractional smoothness can also be obtained for the sphere in the form of an 
Aronszajn-Smith formula using spherical harmonics and a Bourgain-Brezis-Mironescu theorem.

\begin{lem}[apr\`es Aronszajn-Smith] 
\label{lem:AS}
Let $F = \sum Y_k$ where $Y_k$ is a spherical harmonic of degree $k$; then for $0<\beta < 
\min \{1,\, n/2\}$ 
\begin{equation}\label{eq:lemAS}
\begin{split}
&\int_{S^n\times S^n} \frac{|F(\xi) - F(\eta)|^2}{|\xi-\eta|^{n+2\beta}} \, d\xi\,d\eta 
= 2\hat A_\beta \sum_{k=1}^\infty 
\left[ \frac{\Gamma (\frac{n}2 - \beta) \Gamma (\frac{n}2 + \beta+k)}{\Gamma (\frac{n}2 +\beta) 
\Gamma (\frac{n}2 -\beta +k)} -1 \right] \int |Y_k|^2\,d\xi\\
\noalign{\vskip6pt}
&\hskip1in
2\hat A_\beta = 2^{-n-2\beta}\ \frac{\Gamma (n+1)}{\Gamma (\frac{n}2 +1)} \ 
\frac{\Gamma (1-\beta)}{\beta\,\Gamma (\frac{n}2 - \beta)}
\end{split}
\end{equation}
\end{lem}

\begin{proof} 
Observe that by using the calculations for the Hardy-Littlewood-Sobolev inequality on the 
sphere (see \cite{Beckner-Forum97}, page 307) when $0<\lambda <n$
\begin{gather*}
\int_{S^n\times S^n}\frac{|F(\xi) - F(\eta)|^2}{|\xi-\eta|^\lambda} \,d\xi\,d\eta 
= 2A_\lambda \sum_{k=1}^\infty 
\left[ 1 - \frac{\Gamma (n-\frac{\lambda}2)\Gamma (\frac{\lambda}2 +k)}
{\Gamma (\frac{\lambda}2) \Gamma (n-\frac{\lambda}2 +k)} \right] 
\int |Y_k|^2\,d\xi\\
\noalign{\vskip6pt}
A_\lambda = \int |\xi-\eta|^{-\lambda} \,d\xi\,d\eta 
= 2^{-\lambda} \ \frac{\Gamma (n)}{\Gamma (\frac{n}2)} \ 
\frac{\Gamma (\frac{n-\lambda}2)}{\Gamma (n-\frac{\lambda}2)}
\end{gather*}
Since the integral is well-defined for $0< \lambda < n+2$, analytic continuation gives the 
desired result for the Lemma.
\end{proof}

Since this smoothing form precisely captures the Hardy-Littlewood-Sobolev coefficients 
for expansion in spherical harmonics, a new representation can be given for the 
Hardy-Littlewood-Sobolev inequality on the sphere.

\begin{thm}\label{thm:HLS-new rep}
For $1<p<2$, $\lambda = 2n/p'$, $0< \beta < \min \{1,\, n/2\}$ and $q= 2n/(n-2\beta)$ 
with $d\xi =$ normalized surface measure 
\begin{gather} 
\int_{S^n} |F|^2\,d\xi - \bigg( \int_{S^n} |F|^p\,d\xi\bigg)^{2/p}
\le \frac1{2A_\lambda}  \int_{S^n\times S^n} 
\frac{|F(\xi) - F(\eta)|^2}{|\xi-\eta|^\lambda}\, d\xi\,d\eta \label{eq:HLS-new rep1}\\
\noalign{\vskip6pt}
\bigg(\int_{S^n} |F|^q\,d\xi\bigg)^{2/q}  - \int_{S^n} |F|^2\,d\xi 
\le \frac1{2\hat A_\beta} \int_{S^n\times S^n} 
\frac{|F(\xi) - F(\eta)|^2}{|\xi-\eta|^{n+2\beta}} \,d\xi\,d\eta
\label{eq:HLS-new rep2}
\end{gather}
Further setting $\int |F|^2\,d\xi =1$
\begin{gather}
\int_{S^n} |F|^2 \ln |F|\,d\xi
\le \frac{\sqrt{\pi}}2\ 
\frac{\Gamma (\frac{n}2 +1)}{\Gamma (\frac{n+1}2)} 
\int_{S^n\times S^n} \frac{|F(\xi) - F(\eta)|^2}{|\xi-\eta|^n}\, d\xi\,d\eta \label{eq:HLS-new rep3}\\
\noalign{\vskip6pt}
= \sum_{k=1}^\infty \Delta_{k,n} \int_{S^n} |Y_k|^2\, d\xi \nonumber\\
\noalign{\vskip6pt}
 \Delta_{k,n} = \frac{n}2 \sum_{m=0}^{k-1} \frac1{m+\frac{n}2} \nonumber
\end{gather}
\end{thm}

\begin{proof} 
These estimates follow clearly from the Hardy-Littlewood-Sobolev inequality and 
the Aronszajn-Smith representation given above (see the section on ``sharp Sobolev 
inequalities'' in \cite{Beckner-Annals93} and the discussion related to Theorem~1 in 
\cite{Beckner-Forum97}).
\begin{equation}
\bigg( \int_{S^n} |F(\xi)|^q\,d\xi\bigg)^{2/q}
\le \sum_{k=0}^\infty \frac{\Gamma (\frac{n}q) \Gamma(\frac{n}{q'} + k)}
{\Gamma (\frac{n}{q'}) \Gamma (\frac{n}q +k)}\ \int_{S^n} |Y_k|^2\,d\xi
\end{equation}
\end{proof}

\section{Density functional theory and Pitt's inequality}

In density functional theory an object of interest is the Coulomb interaction energy 
\begin{equation}\label{eq:Coulomb}
E_c (\psi) = \int_{\real^n \times \real^n} 
\frac1{|x-y|} |\psi (x,y)|^2\,dx\,dy 
\end{equation} 
Observe that the structure of Pitt's inequality will provide a sharp estimate for an upper 
bound for $E_c (\psi)$ in terms of fractional Sobolev embedding on $\real^{2n}$.
This analysis illustrates the principle that in general product functions may not provide 
an optimal estimation strategy. 

\begin{thm}\label{thm:Density}
For $f\in\S(\real^{2n})$, $0<\lambda < n$, $\lambda = 2\alpha$
\begin{gather} 
\int_{\real^n \times\real^n} |x-y|^{-\lambda} |f(x,y)|^2\,dx\, dy
\le C_\lambda \int_{\real^{2n}} |(-\Delta/4\pi^2)^{\alpha/2} f |^2\,dx\,dy \label{eq:Density}\\
\noalign{\vskip6pt}
C_\lambda = \big( \pi/\sqrt{2}\,\big)^\lambda 
\left[\frac{\Gamma (\frac{n-\lambda}4)}{\Gamma (\frac{n+\lambda}4)}\right]^2
\nonumber
\end{gather}
This   constant is sharp but not attained.
\end{thm}

\begin{proof} 
Inequality \eqref{eq:Density} is equivalent to the fractional integral inequality:
\begin{gather*}
\int_{\real^n\times\real^n} |x-y|^{-\lambda} 
\bigg| \int_{\real^n\times\real^n} \Big[ |x-u|^2 +|y-v|^2\Big]^{-(2n-\alpha)/2} 
f(u,v)\,du\, dv\bigg|^2\, dx\,dy\\
\noalign{\vskip6pt}
\le C_\lambda \int_{\real^n\times\real^n } |f(x,y)|^2\,dx\,dy\\
\noalign{\vskip6pt}
C_\lambda = \pi^{-2n+2\alpha} \left[ \frac{\Gamma (\frac{2n-\alpha}2)}{\Gamma (\frac{\alpha}2)} 
\right]^2 C_\lambda 
\end{gather*}
By duality this is equivalent to 
$$\int_{\real^{2n}} \bigg| \int_{\real^{2n}} |w-z|^{-(2n-\alpha)} |x-y|^{-\lambda/2} 
h(z)\,dz\bigg|^2\,dw 
\le C_\lambda \int_{\real^{2n}} |h(z)|^2\,dz$$
where $z= (x,y) \in \real^n\times\real^n$ and $w = (u,v) \in \real^n\times\real^n$. 
By integrating out the free variable on the left-hand side, the inequality becomes 
\begin{gather*}
\int_{\real^{2n} \times\real^{2n}} h(z)\, |x-y|^{-\lambda/2} |z-w|^{-(2n-2\alpha)} 
|u-v|^{-\lambda/2} h(w)\,dz,dw\\
\noalign{\vskip6pt}
\le H_\lambda \int_{\real^{2n}} |h(z)|^2\,dz\\
\noalign{\vskip6pt}
C_\lambda = \pi^{-n+2\alpha} \Gamma (n-\alpha) /\Gamma (\alpha) H_\lambda\\
\end{gather*}
To analyze the left-hand side, consider the rotation
$$R = \left[ \begin{matrix} \frac1{\sqrt2} \Bone_n & -\frac1{\sqrt2} \Bone_n\\
\frac1{\sqrt2} \Bone_n & -\frac1{\sqrt2} \Bone_n
\end{matrix}\right]$$
where $\Bone_n$ is the identity matrix on $\real^n$ and let $P$ denote the projection 
on the first $n$ variables. 
Then the left-hand side corresponds to 
$$\big(\sqrt2\,\big)^{-\lambda} \int_{\real^{2n}\times\real^{2n}} 
h(z) |PRz|^{-\lambda/2} |z-w|^{-(2n-2\alpha)} |PRw|^{-\lambda/2} h(w)\,dz\,dw\ ;$$
by changing variables, $z\to Rz$ and relabeling with the observation on matrices 
$$R^{-2} = \left[ \begin{matrix} 0&\Bone_n\\ -\Bone_n &0\end{matrix}\right]$$
this term can be rewritten as 
$$\big(\sqrt2\,\big)^{-\lambda} \int_{\real^{2n} \times\real^{2n} }
h(y_1 -x) |x|^{-\lambda/2} |z-w|^{-(2n-2\alpha)} |u|^{-\lambda/2} h(v_1 -u) \,dz\,dw$$
Applying Young's inequality in the variables $y$ and $v$ provides the upper bound 
$$K_\lambda \int_{\real^n\times\real^n} \bar h (x) \, |x|^{-\lambda/2} |x-u|^{-\lambda/2} 
\bar h(u)\,du\,dv$$
where 
\begin{gather*} 
\bar h(x) = \bigg[ \int_{\real^n} \big| h(y, -x)\big|^2\,dy\bigg]^{1/2}\\
\noalign{\vskip6pt}
K_\lambda = \big(\,\sqrt{2}\,\big)^{-\lambda} \int_{\real^n} (1+ |y|^2)^{-(n-\lambda/2)}\,dy 
= 2^{-\lambda/2} \pi^{n/2} \frac{\Gamma(\frac{n-\lambda}2)}{\Gamma (n-\frac{\lambda}2)}
\end{gather*}
Pitt's inequality completes the argument:
\begin{gather*}
\int_{\real^n\times\real^n} \bar h (x) \, |x|^{-\lambda/2} |x-u|^{-(n-\lambda)} |u|^{-\lambda/2} 
\bar h (u)\,du\,dv 
\le E_\lambda \int_{\real^n} |\bar h(x)|^2\,dx\\
\noalign{\vskip6pt}
E_\lambda = \pi^{n/2} \ \frac{\Gamma (\frac{\lambda}2)}{\Gamma (\frac{n-\lambda}2)} \ 
\left[\frac{\Gamma (\frac{n-\lambda}4)} {\Gamma (\frac{n+\lambda}4)}\right]^2
\end{gather*}
and 
$$\int_{\real^n} |\bar h(x)|^2\,dx 
= \int_{\real^{2n}} |h(x,y)|^2\,dx\,dy$$
Tracing through all the steps in calculating the optimal constant gives:
\begin{equation*}
\begin{split}
C_\lambda &= \pi^{-n+\lambda}\ \frac{\Gamma (n-\frac{\lambda}2)}{\Gamma(\frac{\lambda}2)}\ 
H_\lambda\\
\noalign{\vskip6pt}
& = \pi^{-n+\lambda} \ \frac{\Gamma (n-\frac{\lambda}2)}{\Gamma (\frac{\lambda}2)}\ 
K_\lambda \, E_\lambda\\
\noalign{\vskip6pt}
& = \pi^{-n+\lambda} 2^{-\lambda/2} \ \frac{\Gamma (n-\frac{\lambda}2)}{\Gamma(\frac{\lambda}2)}
\left[ \pi^{n/2} \frac{\Gamma (\frac{n-\lambda}2)}{\Gamma (n-\frac{\lambda}2)}\right]
\left[ \pi^{n/2} \frac{\Gamma (\frac{\lambda}2)}{\Gamma (\frac{n-\lambda}2)}\right] 
\left[ \frac{\Gamma (\frac{n-\lambda}4)}{\Gamma (\frac{n+\lambda}4)}\right]^2\\
\noalign{\vskip6pt}
& = (\pi^2/2)^{\lambda/2} \left[ \frac{\Gamma (\frac{n-\lambda}4)}{\Gamma (\frac{n+\lambda}4)}
\right]^2
\end{split}
\end{equation*}
\end{proof}

The calculation above expresses a mixing of radial symmetry and product structure 
which can be outlined in the following lemma. 

\begin{Lem} 
For $f\in \S (\real^n\times\real^m)$, $0<\lambda < n$, $\lambda = 2\alpha$
\begin{gather} 
\int_{\real^n\times \real^m} |x|^{-\lambda} |f(x,y)|^2\,dx\,dy 
\le D_\lambda \int_{\real^n\times\real^m} \big| (-\Delta/4\pi^2)^{\alpha/2} f\big|^2\,dx\,dy 
\label{eq:mixing}\\
\noalign{\vskip6pt}
D_\lambda = \pi^\lambda \left[ \frac{\Gamma (\frac{n-\lambda}4)}{\Gamma(\frac{n+\lambda}4)}
\right]^2
\nonumber
\end{gather}
\end{Lem}

\begin{proof} 
Inequality \eqref{eq:mixing} is equivalent to the fractional integral inequality: 
\begin{gather*}
\int_{\real^n\times\real^m} |x|^{-\lambda} 
\bigg| \int_{\real^n\times \real^m} \big[ |x-u|^2 + |v|^2\big]^{-(n+m - 3\alpha)/2} 
f(u,v)\,du\,dv\bigg|^2 \,dx\,dy \\
\noalign{\vskip6pt}
\le G_\lambda \int_{\real^n\times\real^m} |f(x,y)|^2\,dx\,dy\\
\noalign{\vskip6pt} 
D_\lambda = \pi^{-n-m+2\alpha} 
\left[ \frac{\Gamma (\frac{n+m-\alpha}2)}{\Gamma (\frac{\alpha}2)} \right]^2\, G_\lambda
\end{gather*}
By duality this inequality is equivalent to 
$$\int_{\real^n\times\real^m} \bigg| \int_{\real^n\times\real^m} |w-z|^{-(n+m-\alpha)} 
|x|^{-\lambda/2} h(z)\, dz\bigg|^2 \,dw 
\le G_\lambda \int_{\real^n\times\real^m} |h|^2\, dz$$
where $z  = (x,y) \in \real^n\times\real^m$ and $w= (u,v)\in \real^n\times\real^m$. 
After integrating out the free variable on the left-hand side, the inequality becomes 
\begin{gather*}
\int_{\real^{n+m}\times\real^{n+m}} h(z) |x|^{-\lambda/2} |z-w|^{-(n+m-2\alpha)} 
|u|^{-\lambda/2} h(w)\,dz\,dw\\
\noalign{\vskip6pt}
\le H_\lambda \int_{\real^{n+m}} |h(z)|^2\,dz\\
\noalign{\vskip6pt}
D_\lambda = \pi^{-(n+m)/2 + 2\alpha} \
\Gamma \Big(\frac{n+m}2 - \alpha\Big) \Big/ \Gamma (\alpha) H_\lambda
\end{gather*}
Applying Young's inequality in the variables $y$ and $v$ provides the upper bound for the 
left-hand side
$$K_\lambda \int_{\real^n\times\real^n} \bar h(x)\ |x|^{-\lambda/2} |x-u|^{-(n-\lambda)} 
|u|^{-\lambda/2} \ \bar h (u)\,dx\,du$$
where 
\begin{gather*}
\bar h(x) = \bigg[ \int_{\real^m} |h(x,y)|^2\,dy\bigg]^{1/2}\\
\noalign{\vskip6pt}
K_\lambda = \int_{\real^m} \big( 1+ |y|^2\big)^{-(n+m-\lambda)/2}\,dy  
= \pi^{m/2}\ \frac{\Gamma (\frac{n-\lambda}2)}{\Gamma (\frac{n+m-\lambda}2)}\\
\noalign{\vskip6pt}
\int_{\real^n\times\real^n} \bar h(x)\ |x|^{-\lambda/2} |x-u|^{-(n-\lambda)} 
|u|^{-\lambda/2} \bar h(u)\,dx\, du\\
\noalign{\vskip6pt}
\le \pi^{n/2}\ \frac{\Gamma (\frac{\lambda}2)}{\Gamma (\frac{n-\lambda}2)} \ 
\left[ \frac{\Gamma (\frac{n-\lambda}4)}{\Gamma (\frac{n+\lambda}4)}\right]^2 
\int_{\real^n} \big| \bar h(x)\big|^2\,dx
\end{gather*}
Tracing through all the steps results in 
$$D_\lambda = \pi^\lambda\ \left[ \frac{\Gamma (\frac{n-\lambda}4)}{\Gamma(\frac{n+\lambda}4)}
\right]^2$$
which is independent of the dimension $m$.
\end{proof}

The lemma allows a more direct proof of Theorem~\ref{thm:Density} but the initial 
argument provides better understanding of the technical structure of the proof. 
Moreover that structure suggests that control of forms such as the Coulomb interaction 
energy by fractional smoothing is more an $\real^n$ result than an $\real^{2n}$ result
as given in Theorem~\ref{thm:Density}. 
This characterization is made explicit by combining Pitt's inequality with the 
Aronszajn-Smith formula.

\begin{thm}\label{thm:Pitt with A-S}
For $f\in \S(\real^{2n})$, $0<\lambda < n$
\begin{equation}\label{eq:Pitt with A-S}
\int_{\real^n\times \real^n} |x-y|^{-\lambda} |f(x,y)|^2\,dx\,dy 
\le F_\lambda \int_{\real^n\times\real^n} |\xi -\eta|^\lambda |\widehat f (\xi,\eta)|^2\,d\xi\, d\eta
\end{equation}
For $0<\lambda < \min (2,n)$
\begin{gather} 
\int_{\real^n\times \real^n}\mkern-26mu  |x-y|^{-\lambda} 
|f(x,y)|^2 \,dx\,dy 
\le G_\lambda \int_{\real^{2n}\times\real^{2n}}\mkern-33mu |x-y-u-v|^{-n-\lambda} |f(x,y) - f(u,v)|^2
\, dx\,dy\ du\,dv \label{eq:Pitt with A-S2}\\
\noalign{\vskip6pt}
F_\lambda =  (\pi /2)^\lambda \left[ \frac{\Gamma (\frac{n-\lambda}4)}
{\Gamma (\frac{n+\lambda}4)} \right]^2 ; \quad 
G_\lambda =  ( \pi/2)^{-n/2} \left(\frac{\lambda}4\right) 
\frac{\Gamma (\frac{n+\lambda}2)}{\Gamma (1-\frac{\lambda}2)} 
\left[ \frac{\Gamma (\frac{n-\lambda}4)}{\Gamma (\frac{n+\lambda}4)}\right]^2 
\nonumber
\end{gather}
\end{thm}

\begin{proof}
Using the rotation $R$ from the proof of Theorem~\ref{thm:Density} and Pitt's inequality
\begin{gather*}
\int_{\real^n\times\real^n} |x-y|^{-\lambda} |f(x,y)|^2\,dx\, dy
= 2^{-\lambda/2} \int_{\real^n\times\real^n} |PRz|^{-\lambda} |f(x,y)|^2\,dx\,dy\\
\noalign{\vskip6pt}
= 2^{-\lambda/2} \int_{\real^n\times\real^n} |x|^{-\lambda} |g(x,y)|^2\,dx\,dy
\le 2^{-\lambda/2} D_\lambda \int_{\real^n\times\real^n} |\xi|^\lambda 
|\widehat g(\xi,\eta)|^2\,d\xi\, d\eta\\
\noalign{\vskip6pt}
= 2^{-\lambda} D_\lambda \int_{\real^n\times \real^n} |\xi -\eta|^\lambda 
|\widehat f(\xi,\eta)|^2 \,d\xi \,d\eta
\end{gather*}
where $z= (x,y)$, $g(x,y) = f(R^{-1}z)$ and for $w = (\xi,\eta)$
\begin{align*}
\widehat g(\xi,\eta) & = \int_{\real^{2n}} e^{2\pi iwz} g(z)\,dz
= \int_{\real^{2n}} e^{2\pi iwz} f(R^{-1} z)\,dz\\
\noalign{\vskip6pt}
& = \int_{\real^{2n}} e^{2\pi i (R^{-1}w)z} f(z)\,dz 
= \widehat f(R^{-1} w)\ .
\end{align*}
Here 
$$D_\lambda = \pi^\lambda \left[ \frac{\Gamma (\frac{n-\lambda}4)}{\Gamma (\frac{n+\lambda}4)}
\right]^2 \Longrightarrow 
F_\lambda = (\pi/2)^\lambda \left[ \frac{\Gamma (\frac{n-\lambda}4)}{\Gamma (\frac{n+\lambda}4)}
\right]^2\ .$$
Using the Aronszajn-Smith formula
\begin{gather*}
\int_{\real^n\times \real^n} |\xi|^\lambda |\widehat g (\xi,\eta)|^2\, d\xi\, d\eta 
= \frac1{E_\lambda} \int_{\real^{2n}\times\real^{2n}} 
\frac{|g(x,y) - g(u,v)|^2}{|x-u|^{n+\lambda}}\, dx\,dy\,du\,dv\\
\noalign{\vskip6pt}
= \frac{2^{(n+\lambda)/2}}{E_\lambda}
 \int_{\real^{2n}\times \real^{2n}} 
\frac{|f(x,y) - f(u,v)|^2}{|x-y-u+v|^{n+\lambda}} \,dx\,dy\,du\, dv
\end{gather*}
with 
$$E_\lambda = \frac4{\lambda} \pi^{\frac{n}2 +\lambda} \ 
\frac{\Gamma (1-\frac{\lambda}2)}{\Gamma (\frac{n+\lambda}2)}\ .$$
Tracing back on the varied constants 
\begin{gather*}
\int_{\real^n\times\real^n} |x-y|^{-\lambda} |f(x,y)|^2 \,dx\, dy
\le 2^{-\lambda/2} D_\lambda \int_{\real^n\times \real^n} |\xi|^\lambda 
|\widehat g (\xi,\eta)|^2\, d\xi\, d\eta\\
\noalign{\vskip6pt}
= G_\lambda \int_{\real^{2n}\times\real^{2n}} 
\frac{|f(x,y) - f(u,v)|^2}{ |x-y-u+v|^{n+\lambda}} \, dx\,dy\, du\, dv
\end{gather*} 
with 
$$G_\lambda = (\pi/2)^{-n/2} \left( \frac{\lambda}4\right) 
\frac{\Gamma (\frac{n+\lambda}2)}{\Gamma (1-\frac{\lambda}2)} 
\left[\frac{\Gamma (\frac{n-\lambda}4)}{\Gamma (\frac{n+\lambda}4)} \right]^2$$
\end{proof}

\begin{Cor}[logarithmic uncertainty] 
For $f\in \S (\real^n)$
\begin{gather}
\int_{\real^n \times\real^n} \ln |x-y| \, |f(x,y)|^2\,dx\, dy 
+ \int_{\real^n\times\real^n} \ln |\xi - \eta|\, |\widehat f (\xi,\eta)|^2\, d\xi\, d\eta\nonumber\\
\noalign{\vskip6pt}
\ge D \int_{\real^n\times\real^n} |f(x,y)|^2 \,dx\,dy \label{eq:log-uncertain}\\
\noalign{\vskip6pt}
D= \psi (n/4) - \ln (\pi/2)\ ,\qquad 
\psi = (\ln \Gamma)'\nonumber
\end{gather}
\end{Cor}

\begin{proof} 
Observe that inequality \eqref{eq:Pitt with A-S} is an equality at $\lambda =0$ so 
differentiate the inequality at $\lambda =0$; one can also derive this logarithmic weighted 
form directly from the original logarithmic uncertainty form using the rotation $R$ above
(see section~2 in \cite{Beckner-Forum}). 
\end{proof}

More generally, such inequalities as above extend to multidimensional components. 
For $A\in \real^n$ with $|A|=1$ and $x_k \in \real^n$
\begin{gather}
\int_{\real^{mn}} \ln \big| \sum A_k x_k \big|\ |f(x_1,\ldots, x_m)|^2\,dx 
+ \int_{\real^{mn}} \ln \big| \sum A_k \xi_k\big|\ |\widehat f (\xi_1,\ldots, \xi_m)|^2\,d\xi
\nonumber\\ 
\noalign{\vskip6pt}
\ge \big[ \psi (n/4) - \ln \pi \big] \int_{\real^{mn}} |f(x_1,\ldots, x_m)|^2\,dx \label{eq:general}
\end{gather}

Note that as in the Lemma following Theorem~\ref{thm:Density} the constant depends 
only on the component dimension~$n$.

As one might expect by   association with the Coulomb interaction energy, the functional 
$$\int_{\real^n\times\real^n} |x-y|^{-\lambda} |f(x,y)|^2\,dx\,dy$$
has an underlying conformal invariance: 
let $f(x,y) = (1+|x|^2)^{-n/p} (1+|y|^2)^{-n/p} F(\xi,\eta)$ for $(\xi,\eta) \in S^n\times S^n$ and 
$\frac{\lambda}2 + \frac{2n}p = n$ (e.g., $p =  2n/[n- (\lambda/2)] >2$); then
\begin{gather} 
\int_{\real^n\times\real^n} |x-y|^{-\lambda} |f(x,y)|^2\,dx\,dy 
= C_{\lambda,n} \int_{S^n\times S^n} |\xi-\eta|^{-\lambda} 
|F(\xi,\eta)|^2\,d\xi\, d\eta \label{eq:conformal invari}\\
\noalign{\vskip6pt}
C_{\lambda,n} = 2^\lambda \ \pi^n \Big[ \Gamma (n/2)/\Gamma (n)\Big]^2 \nonumber
\end{gather}
where $d\xi, d\eta$ denote normalized surface measure with the map from $\real^n$ to $S^n$ 
defined by 
\begin{gather*}
\xi = \left( \frac{2x}{1+|x|^2}, \frac{1-|x|^2}{1+|x|^2}\right)\ ,\quad
d\xi = \pi^{-n/2} \Big[ \Gamma(n) /\Gamma (n/2)\Big] \Big( 1+ |x|^2\Big)^{-n}\, dx\\
\noalign{\vskip6pt}
|x-y| = \frac12 |\xi-\eta| \Big[ (1+|x|^2)( 1+ |y|^2)\Big]^{1/2}
\end{gather*}
This invariance is highly suggestive to examine the case of product states where one expects 
that optimal constants will be attained in contrast to the earlier results 
obtained by   relating the functional to Pitt's inequality. 
Then for $f(x,y) = \varphi (x)\varphi(y)$ 

\begin{thm}\label{thm:contrast}
For $\varphi \in \S(\real^n)$ and $p = 2n/[n-(\lambda/2)]$, $0<\lambda <n$ 
\begin{gather} 
\int_{\real^n\times\real^n} \varphi^2 (x) |x-y|^{-\lambda} \varphi^2 (y)\,dx\,dy 
\le b_{\lambda,n} \bigg[ \int_{\real^n} |\varphi |^p\,dx\bigg]^{4/p}\nonumber\\
\noalign{\vskip6pt}
\le b_{\lambda,n}\, d_{\lambda,n} \bigg[ \int_{\real^n} \big| (-\Delta/4\pi^2)^{\lambda/8} \varphi\big|^2
\,dx \bigg]^2 \label{eq:contrast}\\
\noalign{\vskip6pt}
b_{\lambda,n} = \pi^{n(-\frac2p)} \ 
\frac{\Gamma (\frac{2n}p - \frac{n}2)}{\Gamma (\frac{2n}p)} 
\left[ \frac{\Gamma (n)}{\Gamma (\frac{n}2)}\right]^{\frac2p -1}\nonumber\\
\noalign{\vskip6pt}
d_{\lambda,n} = \left[ \pi^{n/p \, - \, n/2} 
\left[ \frac{\Gamma (n/p)}{ \Gamma (n/p')}\right]
\left[ \frac{\Gamma (n)}{\Gamma (n/2)}\right]^{1\, -\, 2/p}\right]^2
\nonumber\\
\noalign{\vskip6pt}
b_{\lambda,n} d_{\lambda,n} = 
\frac{\Gamma (\frac{2n}p - \frac{n}2)}{\Gamma (\frac{2n}p)} \ 
\left[\frac{\Gamma (n/p)} {\Gamma (n/p')}\right]^2\ 
\left[\frac{\Gamma (n)}{\Gamma (n/2)}\right]^{1\, - \, 2/p}\nonumber
\end{gather}
\end{thm}

\begin{proof} 
This result follows from successive applications of the Hardy-Littlewood-Sobolev inequality. 
For the first step, observe that if $q= n/[n-(\lambda/2)]$, then $\lambda = 2n/q'$ which results 
in the first part of inequality \eqref{eq:contrast}; 
the second step is an equivalent form taken from the Lemma in section~6 below.
\end{proof}

As a consequence of inverting the fractional smoothing in this result, one can find an 
equivalent representation in terms of a 
multilinear Hardy-Littlewood-Sobolev inequality.

\begin{thm}\label{thm:inverting}
For $\varphi \in L^2 (\real^n)$ and $0<\lambda <n$
\begin{gather} 
\int_{\real^{6n}} \varphi (u_1) \varphi (u_2) \varphi (v_1) \varphi (v_2) \ 
\prod_k |x-u_k|^{-(n\,-\, \lambda/4)} \ 
\prod_k (|y-v_k)|^{-(n \, -\, \lambda/4)} |x-y|^{-\lambda} \, dx\,dy\  du\,dv \nonumber\\
\noalign{\vskip6pt}
\le A_\lambda \bigg( \int_{\real^n}  |\varphi|^2\,dx\bigg)^2 \label{eq:inverting}
\end{gather}
\end{thm}

\begin{equation*}
A_\lambda = \pi^{-4n/p} \ 
\frac{\Gamma [n(\frac2p - \frac12)]}{\Gamma [\frac{2n}p]} 
\left[ \frac{\Gamma [n(1+\frac2p)/4]}{\Gamma [n(1-\frac2p)/4}\right]^4 
\left[ \frac{\Gamma (n/p)}{\Gamma (n/p')} \right]^2 
\left[ \frac{\Gamma (n)}{\Gamma (n/2)}\right]^{1\, -\, 2/p}
\end{equation*}

The most striking feature of this inequality is that it comes from two successive applications 
of sharp conformally invariant inequalities but results in a form that is not clearly amenable 
to application of symmetry methods to determine extremal functions. 
This obstruction is due to the interior integrals over the $(x,y)$ variables. 
In one case, $\lambda =4$ for $n>4$, an extremal function for inequality \eqref{eq:conformal invari} 
is given by 
$$c (1+|x|^2 )^{-(n/2\, -\, 1)}$$
which then allows the extremal function for inequality \eqref{eq:contrast} to be obtained as 
the solution for the convolution equation 
$$\varphi *   |x|^{-(n-1)} = c(1+ |x|^2)^{-(n/2\, -\, 1)}$$

Determining the physical behavior and mathematical description for many body dynamics is 
generally hard --- because both the complexity of symmetry and the possible combination 
of interaction increase substantially. 
A simple example that results from an application of the Hardy-Littlewood-Sobolev inequality 
and could relate to multiparticle interaction is given by 
\begin{equation}\label{eq:multipartials} 
\Big| \int_{\real^n\times\real^n \times\real^n} |x+y+z|^{-\lambda} 
f(x) f(y) f(z)\,dx\,dy\,dz\Big| 
\le B_\lambda \big[ \|f\|_{L^p(\real^n)} \big]^3 
\end{equation}
for $\lambda =  3n/p'$ and $p' >3$.
But how the optimal constant $B_\lambda$ could be calculated is unclear. 
The critical question to understand here is the character of metrics that span multiple points. 

By adapting Theorem~\ref{thm:Pitt with A-S}  to the case of product functions, a novel 
representation of Coulomb interaction forms is outlined which appears to be formulated 
using the structure of the Hardy-Littlewood-Sobolev inequality but is in fact a realization 
of Pitt's inequality. 
While the most direct proof of this results is obtained from Pitt's inequality, an alternative 
proof can be given using a combination of the Hardy-Littlewood-Sobolev inequality, 
the reverse Hardy-Littlewood-Sobolev inequality and the Hausdorff-Young inequality. 
This appears to be one of the first examples where the reverse Hardy-Littlewood-Sobolev 
inequality has an interesting application. 
Part of this inequality was already used in Carneiro's thesis (see pages 3133-3134 in 
\cite{Carneiro}). 

\begin{thm}\label{thm:HLS reverse}
For $\varphi \in \S (\real^n)$, $0< \lambda < n$
\begin{equation}\label{eq:HLS reverse}
\int_{\real^n\times \real^n} |x-y|^{-\lambda} |\varphi (x)|^2 |\varphi (y)|^2 \,dx\,dy 
\le F_\lambda \int_{\real^n \times\real^n} |\xi-\eta|^\lambda 
|\widehat \varphi (\xi)|^2 |\widehat\varphi (\eta)|^2\,d\xi\,d\eta 
\end{equation} 
\end{thm}

\begin{proof}[Alternative proof (without sharp constants)]
Use the first line of inequality \eqref{eq:contrast} to obtain 
\begin{equation*}
\int_{\real^n\times\real^n} |x-y|^{-\lambda} |\varphi (x)|^2 |\varphi (y)|^2\,dx\,dy 
\le b_{\lambda,n} \bigg[ \int_{\real^n} |\varphi |^p\,dx\bigg]^{4/p}
\end{equation*}
for $p = 2n/(n\, -\, \lambda/2) >2$ using the Hardy-Littlewood-Sobolev inequality; 
now apply the Hausdorff-Young inequality to obtain 
$$\le b_{\lambda,n} c_{p,n} \bigg[ \int_{\real^n} |\widehat \varphi|^{p'}\bigg]^{4/p'} \ ,
\qquad p' = 2n/ \Big(n +\frac{\lambda}2\Big)
$$
and now apply the reverse Hardy-Littlewood-Sobolev inequality to find 
$$\le b_{\lambda,n} c_{p,n} e_{\lambda,n} \int_{\real^n\times\real^n} |\xi-\eta|^\lambda 
|\widehat\varphi (\xi|^2 |\widehat\varphi (\eta)|^2\,d\xi \, d\eta$$
for $0< \lambda < n$ and $p = 2n/(n\, -\, \lambda/2)$.
\end{proof}

\begin{Cor}
For $\varphi \in \S (\real^n)$, $n>2$ and $\Omega (\xi) = \sum_{i<j} |\xi_i - \xi_j|^2$
\begin{equation}\label{eq:Omega} 
\int_{\real^{mn}}  \sum_{i<j} 
|x_i - x_j|^{-2} \prod_{k=1}^m  |\varphi (x_k)|^2 \,dx 
\le \frac{4\pi^2}{(n-2)^2} \int_{\real^{mn}} \Omega (\xi) \prod_{k=1}^m 
|\widehat\varphi (\xi_k)|^2\,d\xi 
\end{equation}
\end{Cor} 

The challenge of extending the Hardy-Littlewood-Sobolev inequality both in terms of 
multiple interaction and retaining ``reverse estimates'' suggests the following inequalities 
that extend equation \eqref{eq:multipartials} for the case $\lambda = mn/p'$, $p' >m$ and 
adapt similar arguments used for the proof of Theorem~\ref{thm:Pitt with A-S}.

\begin{thm}\label{thm:reverse estimates}
For $f\in \S (\real^{mn})$, $0 < \lambda < n$
\begin{gather}
\int_{\real^{mn}} \big|\sum x_k \big|^{-\lambda} \big| f(x_1,\ldots, x_m)\big|^2\,dx 
\le F_\lambda \int_{\real^{mn}} \big|\sum \xi_k \big|^\lambda 
\big|\widehat f (\xi_1,\ldots,\xi_m)\big|^2\, d\xi
\label{eq:reverse estimates1}\\
\noalign{\vskip6pt}
\int_{\real^{mn}} \big|\sum x_k \big|^{-\lambda} \prod \big| \varphi (x_k)\big|^2\, dx
\le F_\lambda \int_{\real^{mn}} \big|\sum \xi_k \big|^\lambda 
\prod \big|\widehat\varphi (\xi_k)\big|^2\, d\xi \label{eq:reverse estimates2}\\
\noalign{\vskip6pt}
F_\lambda = (\pi/m)^\lambda \left[ \frac{\Gamma (\frac{n-\lambda}4)}{\Gamma \frac{n+\lambda}4)}
\right]^2\nonumber
\end{gather}
For $\lambda = mn/p'$, $p' > m$ ($p = \frac{mn}{mn-\lambda}$, $q = \frac{mn}{mn+\lambda}$)
\begin{gather}
\int_{\real^{mn}} \big| \sum x_k\big|^{-\lambda} \prod |f(x_k)|\, dx 
\le c_1 \Big[ \|f\|_{L^p (\real^n)} \Big]^m \label{eq:reverse estimates3}\\
\noalign{\vskip6pt}
\int_{\real^{mn}} \big| \sum \xi_k\big|^\lambda \prod |g(\xi_k)|\, d\xi 
\ge c_2 \Big[ \|g\|_{L^q (\real^n)} \Big]^m \label{eq:reverse estimates4}
\end{gather}
\end{thm}

\begin{proof}
The proof of inequality \eqref{eq:reverse estimates1} follows  the argument used in the proof
of Theorem~\ref{thm:Pitt with A-S} and inequality \eqref{eq:Pitt with A-S}. 
An alternate proof of \eqref{eq:reverse estimates2} follows the method of 
Theorem~\ref{thm:HLS reverse} using the multilinear Hardy-Littlewood-Sobolev 
inequalities \eqref{eq:reverse estimates3} and \eqref{eq:reverse estimates4}. 
The first inequality is obtained by iterating the following reduction so that the 
estimate depends on the case $m=2$ which is the original Hardy-Littlewood-Sobolev 
inequality. 
First, use rearrangement and symmetrization to reduce the problem to the case where $f$ 
is radial decreasing: 
\begin{equation*}
\int_{\real^{mn}} \big| \sum  x_k\big|^{-\lambda} \prod |f(x_k)|\, dx 
\le \int_{\real^{mn}} \big| \sum x_k\big|^{-\lambda} \prod f^* (x_k) \,dx
\end{equation*}
where $f^*$ is the equimeasurable radial decreasing rearrangement of $|f|$ on $\real^n$. 
Then observe that 
\begin{equation*} 
f^* (x) \le c\| f\|_{L^p (\real^n)} |x|^{-n/p}\ ,\qquad 
|x|^{-\lambda} * |x|^{-n/p} = c |x|^{-\lambda + \, n/p'}
\end{equation*}
so that 
\begin{align*}
&\int_{\real^{mn}} \big| \sum x_k\big|^{-mn/p'} \prod f^* (x_k) \, dx_1 \ldots dx_m \\
\noalign{\vskip6pt}
&\qquad
\le c\| f\|_{L^p (\real^n)} \int_{\real^{(m-1)n}} \big| \sum x_k\big|^{- (m-1)n/p'} 
\prod f^* (x_k) \, dx_1 \ldots dx_{m-1}
\end{align*}
Continuing this iteration, one obtains the reduction 
\begin{equation*}
\int_{\real^{mn}} \big| \sum x_k\big|^{-mn/p'} \prod f^* (x_k)\, dx\\
\le c\Big[\| f\|_{L^p (\real^n)} \Big]^{m-2} \int_{\real^n\times\real^n} 
f^* (x) |x-y|^{-2n/p'} f^* (y)\, dx\,dy
\end{equation*}
and the proof of inequality \eqref{eq:reverse estimates3} is obtained by using the 
Hardy-Littlewood-Sobolev inequality. 
The second inequality \eqref{eq:reverse estimates4} is obtained by iterating a similar
reduction to the one just used so that the estimate depends on the case $m=2$ which 
is the reverse Hardy-Littlewood-Sobolev inequality (see appendix). 
Again use rearrangement and symmetrization to reduce the problem to the case where 
$g$ is radial decreasing:
\begin{equation*} 
\int_{\real^{mn}} \big| \sum \xi_k\big|^\lambda \prod | (g(\xi_k)|\, d\xi 
\ge \int_{\real^{mn}} \big| \sum \xi_k\big|^\lambda \prod g^* (\xi_k)\, d\xi 
\end{equation*}
where $g^*$ is the equimeasurable radial decreasing rearrangement of $|g|$ on $\real^n$.
Here one uses the following variation on the Brascamp-Lieb-Luttinger rearrangement 
inequality for the function $h$ being radial and increasing:
\begin{align*}
&\int_{\real^{mn}} h\Big( \sum b_k x_k\Big) \prod_{\ell=1}^N \Big| g_\ell \Big( \sum_k a_{\ell k} 
x_k\Big) \Big| \, dx_1 \ldots dx_m\\
\noalign{\vskip6pt}
&\qquad 
\ge \int_{\real^{mn}} h \Big( \sum b_k x_k \Big) \prod_{\ell=1}^N g_\ell^* 
\Big( \sum_k a_{\ell k} x_k\Big) \, dx_1 \ldots dx_m
\end{align*}
Now observe that 
\begin{align*}
|g^* (x)|^{-\alpha}
& \ge c \Big[ \|g\|_{L^q (\real^n)}\Big]^{-\alpha} |x|^{\alpha n/q}\\
\noalign{\vskip6pt} 
\int_{\real^n} g^* (x) |x-y|^\lambda \, dx 
&\ge c \int_{\real^n} \big[ g^* (x)\big]^{1+\alpha} |x|^{\alpha n/q} |x-y|^\lambda \,dx 
\Big[ \|g\|_{L^q (\real^n)} \Big]^{-\alpha}\\
\noalign{\vskip6pt}
&\ge c \big\| |x|^{\alpha n/q} |x-y|^\lambda \big\|_{L^r (\real^n)} 
\big\| (g^*)^{1+\alpha}\big\|_{L^{r'} (\real^n)} 
\Big[ \|g\|_{L^q (\real^n)}\Big]^{-\alpha} \ ;
\end{align*}
set $r = -s <0$, $r'  = s/(s+1)$ with $(1+\alpha)s/(1+s) = q$. 
Then $\| (g^* )^{1+\alpha}\|_{r'} = (\|g\|_q)^{1+\alpha}$; 
rewriting the relation for $q$ gives $s(1-q) + \alpha s = q$ which implies $\alpha s < q$ since 
$q<1$. 
Using the relation for $q$ with respect to $\lambda$, $q = mn /(mn+\lambda)$, three 
equivalent defining relations for $\alpha$ and $s$ can be given in terms of the input 
value for $\lambda$; 
\begin{equation*}
s (1-q)+ \alpha s = q\ ;\qquad 
\frac1s = (1+\alpha) \frac{\lambda}{mn} +\alpha\ ;\qquad 
\lambda + \alpha\ \frac{n}q - \frac{n}s = (m-1)\lambda/m
\end{equation*}
Any values of $\alpha$ and $s$ can be used in the following calculation as long as 
$\alpha s <q$ and $s\lambda <n$; the first condition holds in general, and the second will 
hold for $\alpha \ge 1$ since then $s<1$, and already $\lambda < n$. 
Then 
\begin{gather*}
\big\| |x|^{\alpha n/q} |x-y|^\lambda \big\|_{L^r (\real^n)}
 \ge \bigg[ \int_{\real^n} |x|^{-s \alpha n/q} |x-y|^{-s\lambda} \,dx\bigg]^{-1/s}\\
 \noalign{\vskip6pt}
 = c\Big[ |y|^{-(s\alpha n/q + s\lambda -n)}\Big]^{-1/s} 
 = c |y|^{\alpha n/q + \lambda - n/s} \\
 \noalign{\vskip6pt}
 = c |y|^{(m-1)\lambda /m} = c |y|^{(m-1)n/p'}\ .
\end{gather*}
Now 
\begin{align*}
&\int_{\real^{mn}} \big| \sum x_k\big|^{mn/p'} \prod g^* (x_k)\, dx_1 \ldots dx_m\\
\noalign{\vskip6pt}
&\qquad 
\ge c \|g\|_{L^q (\real^n)} \int_{\real^{(m-1)n}} \big| \sum x_k\big|^{- (m-1) n/p'} 
\prod g^* (x_k) \, dx_1 \ldots dx_{m-1}
\end{align*}
Continuing this iteration, one obtains the reduction 
\begin{align*}
& \int_{\real^{mn}} \big| \sum x_k\big|^{mn/p'} \prod g^* (x_k)\,dx \\
\noalign{\vskip6pt}
&\qquad 
\ge c \Big[ \|g\|_{L^q (\real^n)} \Big]^{m-2} \int_{\real^n\times\real^n} 
g^* (x) |x+y|^{2n/p'} g^* (y)\, dx\, dy
\end{align*}
and the proof of inequality \eqref{eq:reverse estimates4} is obtained from the reverse 
Hardy-Littlewood-Sobolev inequality.

These two expanded Hardy-Littlewood-Sobolev estimates combined with the 
Hausdorff-Young inequality give a proof without sharp constants for inequality 
\eqref{eq:reverse estimates2}. 
Choose $p$ so   that for $0<\lambda < n$, $\lambda = mn/p'$. 
Then using \eqref{eq:reverse estimates3}, \eqref{eq:reverse estimates4} and 
the Hausdorff-Young inequality for $r= 2mn/(mn-\lambda) = 2p >2$ and 
$r' = 2mn/(mn+\lambda) = 2q <2$
\begin{align*}
& \int_{\real^{mn}} \big| \sum x_k\big|^{-\lambda} \prod |\varphi (x_k)|^2\,dx 
\le c\Big[ \|\varphi \|_{L^r (\real^n)} \Big]^{2m}\\
\noalign{\vskip6pt}
&\qquad 
\le c \Big[ \|\widehat\varphi \|_{L^{r'} (\real^n)} \Big]^{2m} 
\le c \int_{\real^{mn}} \big|\sum \xi_k\big|^\lambda \prod |\widehat\varphi (\xi_k)|^2\, d\xi\ .
\end{align*}
Here $c$ is a generic constant, and the proof of \eqref{eq:reverse estimates2} is complete.
\end{proof}

\section{Hardy-Littlewood-Sobolev inequality}

A natural question that underlines the development described here and in recent papers
--- {\em identify the intrinsic character of the Hardy-Littlewood-Sobolev inequality}. 
The starting point would be the fractional integral defined by the Riesz potential 
\begin{equation}\label{eq:HLS-Riesz}
f\in L^p (\real^n) \rightsquigarrow \frac1{|x|^\lambda} * f \in L^{p'} (\real^n)
\end{equation}
with $1<p <2$, $1/p + 1/p' = 1$ and $\lambda = 2n/p'$. 
Here conformal invariance enables calculation of the sharp constant for the 
operator norm \cite{Lieb}:
\begin{gather}
\Big\| \frac1{|x|^\lambda} * f\Big\|_{L^{p'}(\real^n)} 
\le A_p \|f\|_{L^p (\real^n)} \label{eq:sharp-constant}\\
\noalign{\vskip6pt}
A_p = \pi^{n/p'}\ 
\frac{\Gamma [n(\frac1p - \frac12)]}{\Gamma (n/p)}\ 
\left[ \frac{\Gamma (n)}{\Gamma (n/2)} \right]^{\frac2p -1} \nonumber
\end{gather}
Later it was recognized that an inherent  axial symmetry would lead to an equivalent 
representation on the Liouville-Beltrami model for hyperbolic space and provide a 
quick determination of the extremal functions for the optimal inequality. 
This calculation demonstrated how hyperbolic symmetry is embedded in the conformal 
structure of the Riesz functional (\cite{Beckner-JFAA97}).

Because of the inequality's structure as a map from a space to its dual, one can 
utilize the {\it square-integrable paradigm} to give an equivalent representation for the 
Hardy-Littlewood-Sobolev inequality in terms of fractional smoothness.

\begin{Lem}
For $f\in \S(\real^n)$, $1< p < 2$, $\alpha = n (1/p - 1/2)$ and $1/p + 1/p' =1$ 
\begin{gather} 
\int_{\real^n} |f|^2\, dx \le C_p \bigg[ \int_{\real^n} \Big| - \Delta/4\pi^2)^{\alpha/2} f\Big|^p\, dx
\bigg]^{2/p} \label{eq:smoothness1}\\
\noalign{\vskip6pt}
\bigg[ \int_{\real^n} |g|^{p'}\,dx\bigg]^{2/p'} \le C_p \int_{\real^n} \Big| (-\Delta/4\pi^2)^{\alpha/2}
f\Big|^2\,dx \label{eq:smoothness2} \\
\noalign{\vskip6pt}
C_p = \pi^{n/p' - n/2} \left[ \Gamma (n/p')\Big/ \Gamma (n/p)\right] 
\left[ \Gamma(n) \Big/ \Gamma (n/2)\right]^{2/p\ - 1}\nonumber
\end{gather}
\end{Lem}

And the Hardy-Littlewood-Sobolev 
inequality can be viewed as a positive-definite symmetric bilinear quadratic form:
\begin{equation}\label{eq:pos-def}
\int_{\real^n\times\real^n} f(x) |x-y|^{-\lambda} f(y)\,dx\, dy 
\le A_p \big( \|f\|_p\big)^2\ ,\qquad \lambda = 2n/p'
\end{equation}
These inequalities suggest that the defining structure of the Hardy-Littlewood-Sobolev 
inequality should be equally identified with its representation in terms of fractional smoothness 
rather than simply in terms of the Riesz potential. 
To be more explicit, the Hardy-Littlewood-Sobolev inequality can be understood in terms 
of control determined by fractional smoothness while the role of the Riesz potential may
be most useful in 
calculating formulas for sharp constants to characterize that control.
This perspective provides critical insight for extending both the multilinear character 
and the domain manifold structure 
for which one can calculate sharp constants for the Hardy-Littlewood-Sobolev inequality.

\section*{Appendix} 

\subsection*{1. Explicit calculation for an integral}$\quad$
\smallskip

For $f\in \S (\real^n)$ and $0<\alpha <2$, consider
\begin{gather}
\Big[ (-\Delta/4\pi^2)^{\alpha/2} f\Big] (x) 
= \gamma_{\alpha,n} \int_{\real^n} \frac{f(x) - f(y)}{|x-y|^{n+\alpha}} \,dy \label{eq:explicit}\\
\noalign{\vskip6pt}
\gamma_{\alpha,n} = (\alpha/2) \pi^{-\alpha -n/2} 
\frac{\Gamma (\frac{n+\alpha}2)} {\Gamma (1-\frac{\alpha}2)} 
= 2/D_\alpha\nonumber
\end{gather}
To verify this constant, apply the Fourier transform to this equation
\begin{equation*}
\begin{split}
|\xi|^\alpha \widehat f(\xi) 
& = \gamma_{\alpha,n}\ \F \bigg[ \int_{\real^n} \frac{f(x) -f(y)}{|x-y|^{n+\alpha}} \,dy\bigg]\\
\noalign{\vskip6pt}
& = \gamma_{\alpha,n}\ \F \bigg[ \int_{\real^n} \frac{f(x) -f(x+y)} {|y|^{n+\alpha}}\,dy\bigg]\\
\noalign{\vskip6pt}
& = \gamma_{\alpha,n}\ \int_{\real^n} \frac{[1-e^{-2\pi iy\xi}]} {|y|^{n+\alpha}}\,dy \ \widehat f(\xi)\\
\noalign{\vskip6pt}
& = \gamma_{\alpha,n} \ \bigg[ \int_{\real^n} \frac{[1-e^{-2\pi iy\cdot\hat\eta}]} {|y|^{n+\alpha}}\,dy
\bigg]|\xi|^\alpha \widehat f (\xi)\ ,\qquad |\eta| =1\\
\noalign{\vskip6pt}
(\gamma_{\alpha,n})^{-1}
& = \int_{\real^n} \frac{[1-e^{-2\pi iy\cdot\hat\eta}]} {|y|^{n+\alpha}} \, dy\\
\noalign{\vskip6pt}
& = \int_0^\infty \left[ \frac{2\pi^{n/2}} {\Gamma (n/2)} - 2\pi\, w^{(2-n)/2} J_{(n-2)/2} (2\pi w)\right]
w^{-1-\alpha}\, dw\\
\noalign{\vskip6pt}
& = (2\pi)^{\alpha + n/2} \int_0^\infty \left[ \Big(2^{n/2-1} \Gamma (n/2)\Big)^{-1} 
- w^{(2-n)/2} J_{(n-2)/2} (w)\right] w^{-1-\alpha} \, dw
\end{split}
\end{equation*}
Observe that for real $w$ 
the series expansion for the Bessel function is given by
$$J_\nu (w) = (w/2)^\nu \sum_{k=0}^\infty (-1)^k 
\frac{(w/2)^{2k}} {k!\ \Gamma (\nu + k+1)}$$
so that the integral above can be calculated using ``integration by parts'':
\begin{equation*}
\begin{split}
&\int_0^\infty\left[ \Big( 2^{n/2-1} \Gamma (n/2)\Big)^{-1} -w^{(2-n)/2} J_{(n-2)/2} (w)\right] 
w^{-1-\alpha}\, dw\\
\noalign{\vskip6pt}
&\qquad 
= - \frac1{\alpha} \int_0^\infty \left[ \Big( 2^{n/2-1} \Gamma (n/2)\Big)^{-1} - w^{(2-n)/2 } 
J_{(n-2)/2} (w) \right] \, d (w^{-\alpha})\\
\noalign{\vskip6pt}
&\qquad 
= - \frac1{\alpha} \int_0^\infty w^{-\alpha} \frac{d}{dw} \left[ w^{(2-n)/2} J_{(n-2)/2} (w)\right]\, dw\\
\noalign{\vskip6pt}
&\qquad 
= \frac1{\alpha} \int_0^\infty w^{-\alpha + (2-n)/2} J_{n/2} (w)\,dw\\
\noalign{\vskip6pt}
&\qquad 
= \frac1{\alpha} \ 2^{1-\alpha-n/2} \ \frac{\Gamma (1-\frac{\alpha}2)} 
{\Gamma (\frac{n+\alpha}2)}\ .
\end{split}
\end{equation*}
Hence
$$\gamma_{\alpha,n} = (\alpha/2) \pi^{-\alpha - n/2} \ 
\frac{\Gamma (\frac{n+\alpha}2)} {\Gamma (1-\frac{\alpha}2)} = 2/D_\alpha$$
Evaluation of the integral for the Bessel function is taken from Erdelyi, 
{\it Higher Transcendental Functions} (see Gradshteyn \& Ryzhik, 
{\it Tables of Integrals, Series, and Products}, Academic Press, 1965, page 684, formula~14).

At first glance the appearance of the Aronszajn-Smith constant is unexpected, but 
it follows directly from the formula for real-valued functions:
\begin{equation}\label{eq:real-valued} 
\int_{\real^n \times\real^n} 
\frac{|f(x) - f(y)|^2}{|x-y|^{n+\alpha}}\, dx\,dy 
= 2 \int_{\real^n} f(x) \bigg[ \int_{\real^n} 
\frac{f(x) - f(y)}{|x-y|^{n+\alpha}} \,dy\bigg]\, dx
\end{equation}

Alternative arguments can be given to calculate this integral using Gaussian 
subordination and Green's theorem: for $\eta \in S^{n-1}$
\begin{equation*}
\begin{split}
&\int_{\real^n} \frac1{|w|^{n+\alpha}} \Big(1- e^{-2\pi iw\cdot\eta} \Big)  \,dw \\
\noalign{\vskip6pt}
&\qquad  = \int_{\real^n} \frac1{|w|^{n+\alpha}} (1-\cos 2\pi w\cdot\eta)\,dw\\
\noalign{\vskip6pt}
&\qquad = \frac{\pi^{\frac{n+\alpha}2}}{\Gamma(\frac{n+\alpha}2)} 
\int_{\real^n} (1-\cos 2\pi w\cdot\eta) 
\int_0^\infty t^{\frac{n+\alpha}2 -1} e^{-\pi tw^2}\,dt \\
\noalign{\vskip6pt}
&\qquad = \frac{\pi^{\frac{n+\alpha}2}}{\Gamma (\frac{n+\alpha}2)}
\int_0^\infty t^{\frac{n+\alpha}2 -1} \int_{\real^n} 
(1-\cos 2\pi w\cdot\eta) e^{-\pi tw^2}\,dw\\
\noalign{\vskip6pt}
&\qquad = \frac{\pi^{\frac{n+\alpha}2}}{\Gamma(\frac{n+\alpha}2)} 
\int_0^\infty \mkern-12mu 
t^{\frac{\alpha}2 -1} (1-e^{-\pi/t})\,dt \\
\noalign{\vskip6pt}
&\qquad = \frac{\pi^{\frac{n+\alpha}2}}{\Gamma(\frac{n+\alpha}2)} 
\int_0^\infty \mkern-12mu t^{-\frac{\alpha}2 -1} (1-e^{-t})\,dt\\
\noalign{\vskip6pt}
&\qquad = \frac2{\alpha} 
\frac{\pi^{\frac{n}2 +\alpha}}{\Gamma(\frac{n+\alpha}2)} 
\int_0^\infty t^{-\alpha/2} e^{-t} \,dt 
= \frac2{\alpha} \pi^{\frac{n}2 +\alpha} 
\frac{\Gamma (1-\frac{\alpha}2)}{\Gamma(\frac{n+\alpha}2)}\ .
\end{split}
\end{equation*}
The positivity of the integrands justify the exchange of orders of 
integration using Fubini's theorem. 
A third argument can be given using distribution theory and 
Green's theorem.
\begin{equation*}
\begin{split}
&\int_{\real^n} \frac1{|w|^{n+\alpha}} (1-\cos 2\pi w\cdot\eta)\,dw\\
\noalign{\vskip6pt}
&\qquad = \frac12
\left[ \alpha \Big(\frac{n+\alpha}2 -1\Big)\right]^{-1} 
\int_{\real^n} \Delta \left(\frac1{|w|^{n+\alpha -2}}\right) 
(1-\cos 2\pi w\cdot \eta)\,dw\\
\noalign{\vskip6pt}
&\qquad = \frac12
\left[ \alpha\Big( \frac{n+\alpha}2 -1\Big)\right]^{-1} 
\int_{\real^n} \frac1{|w|^{n+\alpha -2}} \Delta (1-\cos 2\pi w\cdot\eta)\,dw\\
\noalign{\vskip6pt}
&\qquad = 
2\pi^2 \left[\alpha \Big(\frac{n+\alpha}2-1\Big)\right]^{-1} 
\int_{\real^n} \frac1{|w|^{n+\alpha -2}} \cos 2\pi w\cdot \eta\, dw\\
\noalign{\vskip6pt}
&\qquad = 
2\pi^2 \left[\alpha \Big(\frac{n+\alpha}2 -1\Big)\right]^{-1} 
{\mathcal F} \Big[ \frac1{|w|^{n+\alpha -2}}\Big] (\eta)\\
\noalign{\vskip6pt}
&\qquad = 
\frac{2\pi^{\frac{n}2 +\alpha}}{\alpha}\ \  
\frac{\Gamma (1-\frac{\alpha}2)}{\Gamma (\frac{n+\alpha}2)}\ .
\end{split}
\end{equation*}
An independent derivation using Pizzetti's formula and analytic continuation can be 
found in Landkof \cite{L} (see formula (1.1.6) on page~46). 

\subsection*{2. Global embedding and boundary value estimates}$\quad$
\smallskip

Symmetrization on the multiplicative group $\real_+$ allows one to obtain a  direct relation 
between the estimates \eqref{eq:thm4-pf} and \eqref{eq:BA93}, here taken on $\real^{n+1}$.
Consider for $\lambda = (n-1)/2$ and $f = |x|^{-\lambda} g$
\begin{gather*}
\int_{\real^{n+1}} |\nabla f|^2\,dx 
 = \int_{\real^{n+1}} \big|\nabla (|x|^{-\lambda} g)\big|^2\,dx 
 = \int_{\real^{n+1}} |x|^{-2\lambda} \Big|\nabla g - \lambda |x|^{-1}   \hat \i_r\Big|^2\,dx \\
 \noalign{\vskip6pt}
 = \int_{\real^{n+1}} |x|^{-2\lambda +n} 
 \left[ |\nabla g|^2 - 2\lambda |x|^{-1} g\ \frac{\partial g}{\partial r} 
 + \lambda^2 |x|^{-2} |g|^2\right] \, dr\,dv
 \end{gather*}
 View $g$ as a function of $r$ and $\xi\in S^n$ with $\nabla_s$ denoting the spherical 
 gradient, $r= |x|$, $\hat{\i}_r = x/|x|$, and $d\sigma$ being standard surface measure on 
 the unit sphere. 
 Observe that $D = r\frac{\partial}{\partial r}$ is the invariant gradient on $\real_+$. 
 Then the expression above can be rewritten as:
 \begin{equation*}
 \begin{split}
 & \int_{\real_+ \times S^n} \Big[ |Dg|^2 + |\nabla_s g|^2 + \lambda^2 |g|^2\Big]
 \ \frac{dr}r\, d\sigma\\
 \noalign{\vskip6pt}
 &\qquad 
 \ge \int_{\real_+\times S^n} \Big[ |Dg^*|^2 + |\nabla_s g^*|^2 + \lambda^2 |g^*|^2\Big]\ 
 \frac{dr}r \, d\sigma\\
 \noalign{\vskip6pt}
 &\qquad 
 = 2 \int_{\{ 0<r<1\} \times S^n} \Big[ |Dg^*|^2 + |\nabla_s g^* |^2 +\lambda^2 |g^*|^2\Big]\ 
 \frac{dr}r\, d\sigma
 \end{split}
 \end{equation*}
 where $g^*$ denotes for each $\xi \in S^n$ the non-negative equimeasurable symmetric 
 decreasing rearrangement of $|g(r,\xi)|$ away from the ``origin'' $r=1$ and as a function of $r$.
 Now the expression above may be rephrased for $f_{\#} = |x|^{-\lambda} g^*$ as 
 \begin{equation*}
 \begin{split}
 & 2\int_{|x|\le 1} |\nabla f_{\#}|^2\,dx + 4\lambda \int_{|x| \le 1} g^* \ \frac{\partial g^*}{\partial r}
 \, dr\,d\sigma\\
 \noalign{\vskip6pt}
 &\qquad 
 = 2\int_{|x|\le 1} \Big| \nabla (|x|^{-\lambda} g^*)\Big|^2\,dx 
 + 4\lambda \int_{|x|\le 1} g^* \ \frac{\partial g^*}{\partial r}\, dr\,d\sigma\\
 \noalign{\vskip6pt}
 &\qquad 
 = 2 \int_{|x|\le 1} |\nabla f_{\#}|^2\,dx + 2\lambda \int_{S^n} |g^* (\xi)|^2\,d\sigma
 \end{split}
 \end{equation*}
 Then 
 \begin{gather*}
 b_n \int_{\real^{n+1}} |\nabla f|^2\,dx 
 \ge 2b_n \int_{|x|\le 1} |\nabla f_{\#}|^2\,dx + 2\lambda b_n \int_{S^n} |g^* (\xi)|^2\,d\sigma\\
 \noalign{\vskip6pt}
 2\lambda b_n = \frac{(n-1)\pi^{-(n+1)/2}}4 \ \Gamma \Big(\frac{n-1}2\Big) 
 = \left[ 2\pi^{(n+1)/2} \big/ \Gamma \Big(\frac{n+1}2\Big)\right]^{-1} 
 = 1/\sigma (S^n)
 \end{gather*}
 and 
 \begin{equation*}
 \begin{split}
 & b_n \int_{\real^{n+1}} |\nabla f|^2\,dx 
 \ge 2b_n \int_{|x|\le 1} |\nabla f_{\#}|^2\, dx + \int_{S^n} |g^* (\xi)|^2\,d\xi\\
 \noalign{\vskip6pt}
 &\qquad 
 \ge 2b_n \int_{|x|\le 1} |\nabla u_{\#}|^2\,dx + \int_{S^n} |g^* (\xi)|^2\,d\xi\\
 \noalign{\vskip6pt}
 &\qquad 
 \ge \bigg[ \int_{S^n} |g^* (\xi)|^2\,d\xi\bigg]^{2/q} 
 \ge \bigg[ \int_{S^n} |g(\xi)|^2\,d\xi \bigg]^{2/q} 
 = \bigg[ \int_{S^n} |f(\xi)|^2\,d\xi\bigg]^{2/q}
 \end{split}
 \end{equation*}
 where $u_{\#} = $ harmonic extension of $g^* (\xi)$ to the interior of the unit ball and 
 using Dirichlet's principle
 $$\int_{|x|\le 1} |\nabla h|^2\,dx \ge \int_{|x|\le 1} |\nabla u_{\#}|^2\,dx$$
 for any $h$ which is a smooth extension of $g^* (\xi)$ to the interior of the unit ball, and 
 $$g^* (\xi) = \sup_{r>0} |g(r,\xi)| \ge |g(\xi)| = |f(\xi)|\ .$$
 Hence putting all the steps together, inequality \eqref{eq:BA93} obtained by using the 
 dual-spectral form of the Hardy-Littlewood-Sobolev inequality on the sphere $S^n$ 
 (see page~233 in \cite{Beckner-Annals93})
 together with symmetrization on the multilplicative group $\real_+$ results in a 
 second derivation of inequality \eqref{eq:thm4-pf}: 
 \begin{gather*}
 \bigg[ \int_{S^n} |f(\xi)|^q\,d\xi\bigg]^{2/q} \le b_n \int_{\real^{n+1}} |\nabla f|^2\,dx \\
 \noalign{\vskip6pt}
 b_n = \frac14 \ \pi^{-(n+1)/2}\ \Gamma \Big(\frac{n-1}2\Big)\ .
 \end{gather*}
 Still Theorem~\ref{thm4}, from which inequality \eqref{eq:thm4-pf} is obtained, is a 
 more general result as it includes fractional smoothing, and by explicit symmetric 
 extension on $\real_+$ can be used to obtain inequality  \eqref{eq:BA93} for 
 harmonic extension on the unit ball in $\real^{n+1}$. 

\subsection*{3. Proof of the reverse Hardy-Littlewood-Sobolev inequality}$\quad$
\smallskip

The conformal invariant structure of the Hardy-Littlewood-Sobolev inequality can be 
continued across the Lebesgue index $p=1$ where for non-negative functions the 
inequality reverses.

\begin{thm}[reverse Hardy-Littlewood-Sobolev inequality]
\label{thm:reverse HLS}
Let $f,g\in L^p (\real^n)$ with $f,g\ge0$, $0<p<1$ and $\lambda = -2n/p'$ $(p=2n/(2n+\lambda))$;
then 
\begin{gather} 
\int_{\real^n\times\real^n} f(x) |x-y|^\lambda g(y)\,dx\,dy
\ge A_\lambda \|f\|_{L^p (\real^n)} \|g\|_{L^p (\real^n)} \label{eq:reverse HLS}\\
\noalign{\vskip6pt}
A_\lambda = \pi^{\lambda/2}\ \frac{\Gamma (\frac{n+\lambda}2)}{\Gamma (n+\frac{\lambda}2)}
\left[ \frac{\Gamma (n)}{\Gamma (\frac{n}2)}\right]^{1+\, \lambda/n}\nonumber
\end{gather}
with extremal functions given up to conformal automorphism by $A (1+ |x|^2)^{-n/p}$.
\end{thm}

\begin{proof}
Since $|x|^\lambda$ is a radial increasing function, apply symmetrization to obtaion 
\begin{equation*}
\int_{\real^n\times\real^n} f(x) |x-y|^\lambda g(y)\, dx\,dy 
\ge \int_{\real^n\times\real^n} f^* (x) |x-y|^\lambda g^* (y)\,dx\,dy
\end{equation*}
where $f^*,g^*$ denote the equimeasurable radial decreasing rearrangements of $f,g$.
The next step is to reduce the problem to the multiplicative group $\real_+$ or equivalently 
the line $\real$. 
Set $h(u) = |x|^{n/p} f^* (x)$, $k(v) = |y|^{n/p} g^* (y)$ where $u = |x|$, $v= |y|$; then 
the inequality becomes
\begin{equation}\label{eq:reverse HLS-2}
\int_{\real_+ \times\real_+} h(u) k(v) \int_{S^{n-1}\times S^{n-1}} \left[ \frac{u}v + \frac{v}u -
2\xi \cdot \eta\right]^{\lambda/2} \,d\xi\,d\eta\, \frac{du}u \frac{dv}v  
\ge B_\lambda \|h\|_{L^p (\real_+)} \|k\|_{L^p (\real_+)} \ .
\end{equation}
Observe that $h,k$ are bounded so that $h, k\in L^p (\real_+) \cap L^\infty (\real_+)$.
The ``potential'' is now symmetric increasing away from the origin $\{ u=1\}$ on $\real_+$ 
so symmetrization will improve the inequality by diminishing the left-hand side so that 
$h(1/u) = h(u)$ is monotone decreasing for $u>1$ (similarly for $k$). 
This step then implies that inequality \eqref{eq:reverse HLS} is improved if 
(1)~$f$ is radial decreasing, 
(2)~$|x|^{n/p} f(x)$ is decreasing for $|x|>1$, and 
(3)~$f(|x|^{-1}) = |x|^{2n/p} f(|x|)$, all for nonnegative $f$.
These conditions are precisely what is meant by saying that $f$ and $g$ possess 
``inversion symmetry''. 
Set $u = e^t$ and $v = e^s$ so that the working inequality becomes 
\begin{equation}\label{eq:working inequality} 
\int_{\real\times\real} h(t) k(s) \int_{S^{n-1} \times S^{n-1}} 
\big[ \cosh (t-s) - \xi \cdot \eta\big]^{\lambda/2} \, d\xi\, d\eta\, dt\,ds 
\ge C\|h\|_{L^p (\real)} \|k\|_{L^p (\real)}
\end{equation} 
Normalize this expression by setting $\|h\|_{L^p (\real)} = \|k\|_{L^p (\real)}=1$; let 
\begin{equation*} 
\int_{S^{n-1} \times S^{n-1}} \big[ \cosh (t) - \xi\cdot \eta\big]^{\lambda/2}\,d\xi\,d\eta
= J_N (t) + K_N (t)
\end{equation*}
where $J_N$ is supported on $\{ |t| <N\}$ and $K_N $ is supported on $\{|t| \ge N\}$. 
First, show that for $\|h\|_p = \|k\|_p =1$
\begin{equation*}
\inf \bigg[ \int_{\real\times\real} h(t) k(s) \int_{S^{n-1}\times S^{n-1}} \big[ \cosh (t-s) -\xi \cdot\eta
\big]^{\lambda/2} \,d\xi\,d\eta \,dt\, ds\bigg] = C >0\ .
\end{equation*}
This fact follows from the reverse Young's inequality where 
\begin{equation*}
\int_{\real\times\real} h(t) k(s) K_1 (t-s) \,dt\,ds 
\ge \|K_1 \|_{L^{p'/2} (\real)} >0
\end{equation*}
so the positive infimum $C$ exists. 
The second objective is to show the existence of extremals where the infimum is attained. 
Consider sequences $\{ h_m\}$, $\{k_m\}$ with $\|h_m\|_p = \|k_m\| = 1$ so that 
\begin{equation*}
\Lambda_m = \int_{\real\times\real} h_m (t) k_m (s) 
\int_{S^{n-1}\times S^{n-1}} \big[ \cosh (t-s) - \xi \cdot \eta\big]^{\lambda/2} 
\, d\xi \, d\eta\, dt\, dt \xrightarrow[m\to\infty]{}  C
\end{equation*}
and $2C >\lambda_m \ge C$; the functions $h_m$, $k_m$ will be symmetric decreasing 
and uniformly bounded by a multiple of $(1+ |t|)^{-1/p}$ where $0 < p < 1$ so that they 
have a uniform $L^1 (\real)$ majorant. 
Since the functions are decreasing, the Helly selection principle can be applied to choose 
subsequences that converge almost everywhere to functions $h$ and $k$ with 
$\|h\|_p \le 1$, $\|k\|_p \le 1$. 
To simplify notation, the pointwise convergent sequences are now substituted in place 
of the original sequences. 
By Fatou's lemma 
\begin{equation*}
\Lambda_x = \int_{\real\times\real} h(t) k(s) \int_{S^{n-1}\times S^{n-1}} 
\big[ \cosh (t-s) - \xi\cdot\eta\big]^{\lambda/2}\, d\xi\, d\eta\, dt\,ds 
\le \lim \Lambda_m = C
\end{equation*}
Past arguments have used a uniform majorant for the sequential  functions to show that 
limit and integral can be interchanged for the ``potential functional'' (see discussion on 
page~40 in \cite{Beckner-Princeton}, and the proof of Theorem~15 on the multilinear 
Hardy-Littlewood-Sobolev inequality in \cite{Beckner-2013}). 
A contrasting strategy is utilized here where control by the ``potential functional'' 
combined with the monotonicity of the functions and the uniform $L^1$ majorants 
will show that 
\begin{equation*}
\int_\real h_m (t)\ e^{\lambda/2|t|} \,dt < D_1\ ,\qquad 
h_m (t) \le D_2 \ e^{-\lambda/2|t|}
\end{equation*}
and similarly for $\{k_m\}$ which gives
\begin{equation*}
1 = \lim \int |h_m|^p \, dt  = \int |h|^p\, dt\ ,\qquad 
1 = \lim \int |k_m|^p \, dt  = \int |k|^p\,dt \ .
\end{equation*}
Since $h$ and $k$ have unit norms, $\Lambda_* = C = \text{infimum}$ and $h,k$ are 
extremal functions for the reverse Hardy-Littlewood-Sobolev inequality. 
In returning to the $\real^n$ setting, there will be extremal functions $f,g$ with 
inversion symmetry. 
The conformal structure of the Hardy-Littlewood-Sobolev functional will determine allowed 
forms for the extremal functions from which the constant $A_\lambda$ in 
equation~\eqref{eq:reverse HLS} can be calculated. 
Since the functional is bilinear, the most direct approach is use conformal symmetry on the 
sphere $S^n$. 
For $\xi,\eta \in S^n$, let $f(x) = (1+ |x|^2)^{-n/p} F(\xi)$ and $g(y) = (1+ |y|^2)^{-n/p} G(\eta)$; 
then 
\begin{gather*} 
\int_{\real^n\times\real^n} |x-y|^\lambda f(x) g(y) \,dx\,dy 
= C_\lambda \int_{S^n\times S^n} |\xi-\eta|^\lambda F(\xi) G(\eta)\,d\xi\, d\eta\\
\noalign{\vskip6pt} 
C_\lambda = 2^{-\lambda}\ \pi^n \Big[ \Gamma (n/2)\big/ \Gamma (n)\Big]^2
\end{gather*}
where $d\xi,d\eta$ denote normalized surface measure on $S^n$ with the map from 
$\real^n$ to $S^n$ defined by 
\begin{gather*}
\xi = \left( \frac{2x}{1+|x|^2}, \frac{1-|x|^2}{1+|x|^2}\right)\ ,\qquad 
d\xi = \pi^{-n/2} \Big[ \Gamma (n) \big/ \Gamma (n/2)\Big] \big(1+|x|^2\big)^{-n} \,dx\\
\noalign{\vskip6pt} 
|x-y| = \frac12 |\xi-\eta] \ \Big[ \big(1+|x|^2\big)\big(1+|y|^2\big)\Big]^{1/2}
\end{gather*}
Then inequality \eqref{eq:reverse HLS} has an equivalent formulation on the $n$-dimensional 
sphere:
\begin{gather}
\int_{S^n\times S^n} F(\xi) |\xi-\eta|^\lambda G(\eta) \,d\xi\, d\eta
\ge B_\lambda \|F\|_{L^p (S^n)}\|G\|_{L^p (S^n)} \label{eq:working inequality 2}\\
\noalign{\vskip6pt}
B_\lambda = \int_{S^n} |\xi-\eta|^\lambda\,d\xi 
= 2^\lambda\ \frac{\Gamma (n)}{\Gamma (\frac{n}2)} \ 
\frac{\Gamma (\frac{n+\lambda}2)}{\Gamma (n+ \frac{\lambda}2)} \nonumber
\end{gather} 
Because only two functions are involved, two-point symmetrization can be used to show 
that the inequality must be improved by rearranged functions that depend only on the 
polar angle and are decreasing away from a pole. 
But the inequality cannot be improved so the extremal functions at this stage must 
combine two properties: 
a)~monotone decreasing away from a pile; 
b)~possess ``inversion symmetry'' which on the sphere means that functions are 
symmetric with respect to an equator. 
Then up to conformal automorphism, the only possible extremals on $S^n$ are constant. 
And this remark completes the proof of Theorem~\ref{thm:reverse HLS}. 
For the dimension $n$ at least two, an alternative determination of the form of the extremals
can be obtained using the hyperbolic symmetry of the Hardy-Littlewood-Sobolev functional. 
For $2 \le \ell \le n$ with the Poincar\'e distance and left-invariant Haar measure on 
$\HH^\ell$ for $w = (x,y) \in \real^{\ell-1} \times \real_+$
\begin{equation*} 
d(w,w') = \frac{|w-w'|}{2\sqrt{yy'}}\ ,\qquad 
d\nu = y^{-\ell} \,dy\,dx 
\end{equation*}
an inequality equivalent to \eqref{eq:reverse HLS} is given by 
\begin{equation} \label{eq:Poincare}
\begin{split}
\int_{\HH^\ell \times \HH^\ell} 
 F(w) G(w')   &\int_{S^{n-\ell} \times S^{n-\ell}} 
\Big[ d^2 (w,w') + (1-\xi\cdot \xi')\big/ 2\Big]^{\lambda/2} \, d\xi\, d\xi'\,d\nu\, d\nu' \\
\noalign{\vskip6pt}
&\qquad 
\ge D_\lambda \|F\|_{L^p (\HH^\ell)} \|G\|_{L^p (\HH^\ell)}
\end{split}
\end{equation}
The constraint of possessing radial symmetry on $\real^n$ in \eqref{eq:reverse HLS} 
and geodesic radial symmetry on $\HH^\ell$ in \eqref{eq:Poincare} will determine 
the form of the extremals (see the argument given in \cite{Beckner-JFAA97} concerning 
``axial symmetry and $SL(2,R)$'' and the proof of Theorem~15 in \cite{Beckner-2013} 
for the multilinear Hardy-Littlewood-Sobolev inequality. 
\end{proof}

\section*{Acknowledgement}

I would like to thank Guozhen Lu for his warm encouragement, and to acknowledge that 
my interest in mathematical questions that lie on the boundary with describing the 
behavior of dynamical phenomena has been encouraged in varied ways by Nestor Guillen, 
Emanuel Carneiro, Luis Caffarelli and Alessio Figalli.


\end{document}